\documentclass[12pt]{amsart}
\usepackage{amscd}
\usepackage{amssymb,amsmath,amsrefs}
\usepackage{hyperref}
\usepackage{mathrsfs}
\usepackage{graphicx}
\usepackage{enumitem}
\usepackage{romannum}
\usepackage{mathtools}
\usepackage[utf8]{inputenc}
\usepackage{listings}
\usepackage{lipsum,eso-pic,xcolor}
\usepackage{lineno}
\usepackage[capitalise]{cleveref}
\usepackage{tikz}
\usepackage{caption}
\usetikzlibrary{arrows,decorations.pathmorphing,backgrounds,positioning,fit}
\usetikzlibrary{positioning}
\hypersetup{
	colorlinks=true,
	linkcolor=blue,}

\usetikzlibrary{trees}
\usetikzlibrary{arrows}

\def\reg{\operatorname{reg}}

\def\deg{\operatorname{deg}}

\def\max{\operatorname{max}}

\def\v{\operatorname{v}}

\newcommand{\x}{\mathbf{x}}

\newcommand{\ZZ}{\mathbb{Z}}
\newcommand{\K}{\mathbb{K}}

\newcommand{\A}{\mathcal{A}}

\newtheorem{lemma}{Lemma}[section]

\newtheorem{theorem}[lemma]{Theorem}
\newtheorem{proposition}[lemma]{Proposition}

\theoremstyle{definition}
\newtheorem{definition}[lemma]{Definition}
\newtheorem{remark}[lemma]{Remark}
\newtheorem{example}[lemma]{Example}

\advance\headheight1.15pt

\numberwithin{equation}{section}
\begin{document}
	
\pagenumbering{arabic}
	
\title[Comparing the v-number and h-polynomial]{Comparing 
the $\mathrm{v}$-number and $h$-polynomials of edge ideals}

\author[Kamalesh Saha]{Kamalesh Saha}
\address[K. Saha]{Department of Mathematics, SRM University-AP, Amaravati 522240, Andhra Pradesh, India}
\email{kamalesh.s@srmap.edu.in; kamalesh.saha44@gmail.com}

\author[Adam Van Tuyl]{Adam Van Tuyl}
\address[A. Van Tuyl]
{Department of Mathematics and Statistics\\
McMaster University, Hamilton, ON, L8S 4L8, CANADA}
\email{vantuyla@mcmaster.ca}

\keywords{v-number, degree of $h$-polynomial, 
Castelnuovo-Mumford regularity, edge ideal of a graph}
\subjclass[2020]{05E40, 13D02, 13D40, 13F55, 05C69}
	
\begin{abstract}
In this paper, we compare the $\mathrm{v}$-numbers and the degree of the $h$-polynomials associated with edge ideals of connected graphs. We prove that the $\mathrm{v}$-number can be arbitrarily larger or smaller than the degree of the $h$-polynomial for the edge ideal of a connected graph.
We also establish that for any pair of positive integers $(v,d)$ with $v \leq d$, there exists a connected graph $H(v,d)$ with the $\mathrm{v}$-number equal to $v$ and the degree of $h$-polynomial equal to $d$. Additionally, we show that the sum of the $\mathrm{v}$-number and the degree of the $h$-polynomial is bounded above by $n$, the number of vertices of $G$, and we classify all graphs for which this sum is exactly $n$. Finally, we show that all thirteen possible inequalities among the three invariants, the $\mathrm{v}$-number, the degree of the $h$-polynomial, and the Castelnuovo-Mumford regularity, can occur in the case of edge ideals of connected graphs. Many of these examples rely on a minimal example of a graph whose $\mathrm{v}$-number is more than the degree of its $h$-polynomial. Using a computer search, we show that there are exactly two such graphs on 11 vertices and 25 edges, and no smaller example on fewer vertices, or 11 vertices and less than 25 edges.
\end{abstract}

\maketitle 

\section{Introduction}
Let $R=\K[x_1,\ldots,x_n]=\bigoplus_{d\geq 0}R_d$ be a polynomial ring in $n$ variables over a field $\K$ with the standard gradation. Let $I$ be a proper graded ideal of $R$. Then any $\mathfrak{p}\in\mathrm{Ass}(I)$ is of the form $I:(f)$ for some homogeneous $f\in R$, where $\mathrm{Ass}(I)$ denotes the set of all associated primes of $I$. This fact allows
us to define the \textit{$\v$-number} of $I$ as follows:
\[
    \v(I) := \min\{ d \geq 0 \mid \text{ there exists }\, f \in R_d \text{ and } \mathfrak{p} \in \text{Ass}(I) \text{ satisfying } I :(f) = \mathfrak{p} \}.
\]

The $\v$-number was introduced in 2020 \cite{cstpv20} to study the asymptotic behavior of the minimum distance function of projective Reed-Muller-type codes. This invariant has since spurred considerable interest in commutative algebra, leading to several significant applications. There has been substantial progress in understanding this invariant across different contexts (see \cites{as24,ass23,bm23,bms24,civan23, concav23, djs25, fd24, fs23, fs24, gp25, grv21, v-edge, js24_v-binom, kmt25, kns25, sahacover23, ssvmon22, ksvgor23, saha24_binomexpan}).

A fundamental tool in commutative algebra is the Hilbert series,
which encodes information about the graded structure of a quotient ring. The {\it Hilbert series} of $R/I$, denoted by $H_{R/I}(t)$, is defined as:
\[
H_{R/I}(t) := \sum_{i \geq 0} \dim_{\mathbb{K}}(R/I)_i \, t^i.
\]
The Hilbert series can also be computed from the minimal free resolution of $R/I$, in particular, from its graded Betti numbers. By the Hilbert-Serre theorem, $H_{R/I}(t)$ can be written as a reduced rational function:
\[
H_{R/I}(t) = \frac{h_{R/I}(t)}{(1-t)^{\dim(R/I)}},
\]
where $h_{R/I}(t) \in \mathbb{Z}[t]$ and $h_{R/I}(1) \neq 0$,
and $\dim(R/I)$ is the Krull dimension. The polynomial $h_{R/I}(t)$ is referred to as the {\it $h$-polynomial} of $R/I$.

The degree of the $h$-polynomial has recently become a central area of interest, particularly in relation to other homological invariants. Since both the Castelnuovo-Mumford regularity (in short, regularity) and the degree of the $h$-polynomial are tied to the minimal graded free resolution, researchers have sought connections between these two invariants
for specific classes of ideals, for example, toric ideals of graphs
\cite{bvt2023, fkvt2020}, monomial ideals \cite{hm2018}, and
binomial edge ideals \cite{hm2022}.

Given a finite simple graph $G$, let $I(G)$ denote its edge ideal, a quadratic square-free monomial ideal whose generators correspond to the edges of $G$. The class of edge ideals of simple graphs is among the nicest classes of square-free monomial ideals. These ideals have been explored extensively over the last four decades. However, it has only been recently that the degree of the $h$-polynomial of $R/I(G)$ and how this invariant compares to other invariants has been studied. For example, a comparison between the degree of the $h$-polynomial and the regularity of edge ideals can be found in \cite{bkoss24, hkkmvt2021, hkmvt2022,hmv19}.

On the other hand, the $\v$-number has emerged as an invariant with many properties similar to the regularity, and for many families of edge ideals, it serves as a sharp lower bound for the regularity. However, as shown by Civan \cite{civan23}, for a connected graph $G$, $\v(I(G))$ can be arbitrarily larger than $\mathrm{reg}(R/I(G))$, although this phenomenon appears rare. This leads to the intriguing question of how $\v(I(G))$ compares to $\deg(h_{R/I(G)}(t))$, given that both invariants are currently of significant interest and share connections to the regularity.

This paper focuses on this question by exploring the relationship between $\v(I(G))$ and $\deg(h_{R/I(G)}(t))$ for connected graphs.  Our first main result establishes
that the difference between these two invariants
can be arbitrary:

\begin{theorem}[{Theorems \ref{Thm:d<v} and \ref{thm:v<=d}}]
For every integer $m \in \mathbb{Z}$, there exists a connected graph $G$ 
such that 
    $$\v(I(G))-\mathrm{deg}(h_{R/I(G)}(t))=m.$$
\end{theorem}

\noindent
Theorem \ref{Thm:d<v} deals with the case that $m$ is positive.
Our Theorem \ref{thm:v<=d} actually shows a stronger
result: for any pair of positive integers $(v,d)$ with
$1 \leq v \leq d$, there exists a connected graph
$G$ with $\v(I(G))=v$ and $\deg(h_{R/I(G)}(t)) =d$. Our
proof involves explicitly constructing a graph with these two invariants. 

The key challenge in this paper is to construct a connected graph $G$ for which $\v(I(G)) > \deg(h_{R/I(G)}(t)),$ and to show that 
the $\v$-number can be arbitrarily larger than the degree of $h$-polynomial in case of edge ideals of connected graphs. In \Cref{sec3}, we establish two constructive lemmas, \Cref{lem:deg} and \Cref{lem:v-num}, where we use $n$ disjoint graphs $G_1, \ldots, G_n$ to build a connected graph $H$. We express the $\v$-number, dimension, and degree of the $h$-polynomial of $H$ in terms of those of the graphs $G_i$. These lemmas may provide a useful tool for computing the $\v$-number and $h$-polynomial of larger graphs by breaking them into smaller pieces. By finding a ``base case" of a graph with  $\v(I(G))>\deg(h_{R/I(G)}(t))$, we can repeatedly apply these lemmas to make the difference between these two invariants as large as possible.  

For our ``base case", we make use of a graph on 11 vertices
and 25 edges that appeared in Jaramillo and Villarreal's paper \cite{v-edge}. This example was used to demonstrate that the $\v$-number could be larger than the regularity of an edge ideal. This
same graph also shows that the $\v$-number can be larger than 
$\deg(h_{R/I(G)}(t))$. Through an extensive computer search, we
show that this example is, in fact, one of two minimal graphs (in terms of the dictionary order on $(\vert V(G)\vert, \vert E(G)\vert)$) satisfying this inequality. 

In \Cref{sec4}, we investigate the sum of the $\v$-number and the degree of $h$-polynomials of edge ideals of graphs. The main theorem of this section is the following:

\begin{theorem}[{Theorem \ref{thm:sumvdeg}}]
Let $G$ be a simple graph on $n$ vertices. Then,
    $$\v(I(G))+\deg(h_{R/I(G)}(t))\leq n.$$
    Moreover, $\v(I(G))+\deg(h_{R/I(G)}(t))= n$ if and only if $G$ is a disjoint union of star graphs.
\end{theorem}
\noindent
This result complements similar results 
on the sum
${\rm reg}(I(G))+\deg(h_{R/I(G)}(t))$ due to
Hibi, Matsuda, and Van Tuyl 
\cite{hmv19} and Biermann, Kara, O'Keefe, Skelton, and Sosa \cite{bkoss24}.

In \Cref{sec5}, we compare three invariants: the $\v$-number, the degree of the $h$-polynomial, and the regularity for edge ideals of connected graphs. We show in Examples \ref{ex:5.1} -- \ref{ex:5.11} that all $13$ possible inequalities among these invariants can occur.  As appendices, we provide tables showing all possible tuples $(\v(I(G)),\deg(h_{R/I(G)}(t)))$ for connected graphs $G$ containing up to $10$ vertices,
 and we provide {\it Macaulay2} code for our two minimal examples of
 graphs whose $\v$-number is larger than the degree of its $h$-polynomial.


\section{Preliminaries}\label{secpreli}

In this section, we recall the relevant definitions and results from 
graph theory and commutative algebra that will be used throughout the paper.

\subsection{Graph Theory} Throughout this paper, every graph is considered to be finite, simple, and non-empty, i.e., a graph with a finite vertex set that has no loops or multiple edges between vertices and also has a non-empty edge set. 
In particular, let $G=(V(G),E(G))$ denote a finite simple graph with
{\it vertex set} $V(G)$ and {\it edge set} $E(G)$. 

A subset $W$ of $V(G)$ is called an {\it independent set} of $G$ if there 
are no edges among the vertices of $W$. An independent set $W$ of $G$ is 
a {\it maximal independent set} if $W \cup \{x\}$ is not an independent 
subset of $G$ for every vertex $x \in V(G) \setminus W$. The maximum size 
among all the independent sets of $G$ is said to be the 
{\it independence number} of $G$ and is denoted by $\alpha(G)$. 
A subset of vertices $C \subseteq V(G)$ is a {\it vertex cover} of $G$ 
if $C\cap e\neq \emptyset$ for all $e\in E(G)$. A {\it minimal vertex cover} is a vertex cover which is minimal with respect to set inclusion. 
The minimum size among all the vertex covers of $G$ is the
{\it vertex covering number} of $G$ and is denoted by $\beta(G)$. 
From the definitions, it follows that $A\subseteq V(G)$ is a maximal 
independent set of $G$ if and only if $V(G)\setminus A$ is a minimal vertex 
cover of $G$. Hence, $\alpha(G)+\beta(G)=\vert V(G)\vert$. 

For $A\subseteq V(G)$, the {\it neighbour set} of $A$ in $G$ is defined
to be 
$$N_{G}(A):=\{x\in V(G)\mid \{x,y\}\in E(G) \text{ for some } y\in A\}.$$ 
Observe that if $A$ is a maximal independent set of $G$, then $N_{G}
(A)=V(G)\setminus A$, which is a minimal vertex cover of $G$.

For a subset $W \subseteq V(G)$, the {\it induced subgraph} of $G$
on $W$ is the graph $G_W = (W,E(G_W))$ where $E(G_W) = \{e \in E(G) ~|~ 
e \subseteq W \}$. In other words, $G_W$ is the subgraph of $G$ obtained
by only considering the edges with vertices in $W$. For $A\subseteq V(G)$, we write $G\setminus A$ to mean the induced subgraph $G_{V(G)\setminus A}$. We now recall some special families of graphs used in this paper.

\begin{definition}
    A graph on $n$ vertices is a \textit{cycle} of length $n$, denoted by $C_n$, if it is connected and every vertex has exactly two neighbours. A graph $G$ is  \textit{chordal} if $G$ has no induced cycle of length more than three. A \textit{complete} graph on $n$ vertices, denoted by $K_n$, is a graph such that there is an edge between every pair of vertices.
\end{definition}

\subsection{Commutative Algebra} The definition of the $\v$-number and the degree of the $h$-polynomial were given in the introduction. We now
define the (Castelnuovo-Mumford) regularity, the third invariant of
interest.

\begin{definition}
    If $I\subseteq R$ is a homogeneous ideal, then the graded minimal free resolution of $R/I$ is a long exact sequence of finitely generated free $R$-modules
\[
\mathcal{F}_{\cdot}: \,\, 0\rightarrow F_r\xrightarrow{\partial_{r}} 
F_{r-1}\rightarrow\cdots\rightarrow F_1\xrightarrow{\partial_1} 
F_0\xrightarrow{\partial_0} R/I\rightarrow 0, 
\]
where $F_0=R$, $F_i=\bigoplus_{j\in\mathbb N}R(-j)^{\beta_{i,j}(R/I)}$, and 
$\partial_i(F_i)\subseteq (x_1,\ldots,x_n)F_{i-1}$ for all $i\geq 1$. 
The numbers $\beta_{i,j}(R/I)$ are uniquely determined and are called 
the $(i,j)^{th}$ {\it graded Betti numbers} of $R/I$. The 
{\it Castelnuovo-Mumford regularity} (or simply the 
{\it regularity}) of $R/I$ is the number
$\reg(R/I):= \max\{j-i\mid \beta_{i,j}(R/I)\neq 0\}.$
\end{definition}

The results of the next lemma, which follows from \cite[Lemma 1.5, Lemma 3.2]{ht10}, and \cite[Proposition 3.9]{ssvmon22}, are well-known; we sometimes apply these facts without reference.

\begin{lemma}\label{lem:reg-deg-v}
        Let $I_1\subseteq R_1=\mathbb K[x_1,\ldots,x_m]$ and 
        $I_2\subseteq R_2=\mathbb K[x_{m+1},\ldots, x_n]$ be two graded ideals. Consider the ideal 
        $I=I_1R+I_2R\subseteq R=\mathbb K[x_1,\ldots,x_n]$. Then 
        \begin{enumerate}
            \item[(i)] $\reg(R/I)=\reg(R_1/I_1)+\reg(R_2/I_2)$;
            \item[(ii)] $H_{R/I}(t)=H_{R_1/I_1}(t)\cdot H_{R_2/I_2}(t)$;
            \item[(iii)] $\deg(h_{R/I}(t))=
            \deg(h_{R_1/I_1}(t))+\deg(h_{R_2/I_2}(t))$; and
            \item[(iv)] $\v(I)=\v(I_1)+\v(I_2)$ if $I_1$ and $I_2$ are monomial ideals.
        \end{enumerate}
\end{lemma}

For a set of variables $A$ in $R$, we write $\mathbf{x}_{A}$ to denote the product of the variables in $A$, i.e., $\mathbf{x}_{A}=\prod_{x_i\in A} x_i$ is a square-free monomial in $R$.

\begin{definition}
    Let $G$ be a graph on the vertex set $V(G)=\{x_1,\ldots,x_n\}$. 
    The \textit{edge ideal} of $G$ is the square-free quadratic monomial ideal of $R$ defined by $I(G):=(\mathbf{x}_{e} \mid e\in E(G)).$
\end{definition}

Jaramillo and Villarreal \cite{v-edge} gave a combinatorial description 
of the $\v$-number for square-free monomial ideals. Since our paper is restricted to the edge ideals of graphs, we rewrite some elementary 
results on the $\v$-number from \cite{v-edge} in graph theory
language. Let $\A_{G}$ be the collection of those independent sets $A$ of $G$ such that $N_{G}(A)$ is a minimal vertex cover of $G$. Note that if 
$A$ is an independent set of $G$ such that $N_{G}(A)$ is a vertex cover 
of $G$, then $N_{G}(A)$ is in fact a minimal vertex cover of $G$. 
The results in the lemma below follow directly from 
\cite[Lemma 3.4 and Theorem 3.5]{v-edge}.  In the statement, 
we abuse notation and write $(N_G(A))$ for the
ideal generated by the vertices of $N_G(A)$, but viewed
as variables in $R$, i.e., $(N_G(A)) = (x_i ~|~ x_i \in N_G(A))$.

\begin{lemma}
Let $G$ be a simple graph. Then the following hold:
\begin{enumerate}
    \item[(i)] if $A\in\A_{G}$, then $I(G):(\x_{A})=(N_G(A))$;
    \item[(ii)] if $I(G):(f)=\mathfrak{p}$ for some 
    $f\in R_{d}$ and $\mathfrak{p}\in 
    \mathrm{Ass}(I(G))$, then there exists $A\in\A_{G}$ such that $\vert 
    A\vert\leq d$ and $I(G):(\x_{A})=(N_G(A))=\mathfrak{p}$; and 
    \item[(iii)] $\v(I(G))=\min\{\vert A\vert\mid A\in \A_G\}$.
\end{enumerate}
\end{lemma}


\section{The v-number and degree of the h-polynomial}\label{sec3}

In this section, our goal is to compare the $\v$-number and the 
degree of the $h$-polynomial for edge ideals of graphs. 
In particular, we demonstrate that for any integer $m \in \mathbb{Z}$,
there is a (connected) graph $G$ such that $\v(I(G)) - \deg (h_{R/I(G}(t)) = m$. We tackle this problem in two stages. We first describe how
to construct a graph $G$ with $\v(I(G)) - \deg h_{R/I(G)}(t) = m$
with $m \geq 1$.  We then focus on showing the existence
of a graph with $\deg (h_{R/I(G)}(t)) 
- \v(I(G)) = m$ with $m\geq 0$.    

The first stage is the more subtle of the two. Civan \cite{civan23} showed
that to compare the $\v$-number and the regularity of edge ideals, the most challenging aspect was to establish that the 
$\v$-number can be arbitrarily larger than the regularity. 
To show the existence of such a graph, Civan required a base case involving
a graph constructed from a triangulation of the dunce cap
consisting of 17 vertices to show $\v(I(G)) 
-\reg(R/I(G)) =1$, while a difference of 
$m$ required a graph with $18(m+1)$ vertices for $m\geq 2$. A similar challenge arises in our comparison -- we require
a graph $G$ with
$\v(I(G)) - \deg (h_{R/I(G)}(t)) = 1$ to act as our ``base case'' 
from which we build our other examples. 

While Civan's result required a graph on 17 vertices,  
Jaramillo and Villarreal \cite{v-edge} showed
the existence of a graph $G$ on 11 vertices and 25 edges
for which $\v(I(G)) > \reg(R/I(G))$, 
provided the characteristic of $\mathbb{K}$ is zero (the two values are the 
same if ${\rm char}(\mathbb{K}) =2$). This graph $G$ is drawn in 
\Cref{fig:11vertex}.  Surprisingly, this graph also has the
property that we require since $\v(I(G)) = 3$ and $\deg(h_{R/I(G)}(t)) = 2$, as checked using {\it Macaulay2}. Since the degree of the $h$-polynomial and the $\v$-number of edge ideals both do not depend on the characteristic of the field, Jaramillo and Villarreal's example also provides us with a characteristic-free example. Using the graph of 
\Cref{fig:11vertex}, we will construct a family with $\v(I(G)) - \deg (h_{R/I(G)}(t)) = m$ for any integer $m \geq 1$. 
In fact, as we show in Section 5, we can use the graph of Figure
\ref{fig:11vertex} to compare the v-number, the degree of the $h$-polynomial, and the regularity.

\begin{figure}[ht]
    \centering
\begin{tikzpicture}[line cap=round,line join=round,>=triangle 45,x=1.5cm,y=1cm]

\draw (0,0)-- (0,2);
\draw (-0.5,1)-- (0,0);
\draw (-0.5,1)-- (0,2);
\draw (2,0)-- (2,2);
\draw (2.5,1)-- (2,0);
\draw (2.5,1)-- (2,2);
\draw (0,0)-- (2,0);
\draw (0,2)-- (2,2);
\draw (-0.5,1)-- (2.5,1);
\draw (-0.5,-1.5)-- (0.5,-1.5);
\draw (1.5,-1.5)-- (2.5,-1.5);
\draw (-0.5,-1.5)-- (-0.5,1);
\draw (-0.5,-1.5)-- (2.5,1);
\draw (-0.5,-1.5)-- (0,0);
\draw (0.5,-1.5)-- (0,2);
\draw (0.5,-1.5)-- (2,2);
\draw (1.5,-1.5)-- (0,0);
\draw (1.5,-1.5)-- (2,0);
\draw (2.5,-1.5)-- (-0.5,1);
\draw (2.5,-1.5)-- (2.5,1);
\draw (2.5,-1.5)-- (0,2);
\draw (1.5,-1.5)-- (1,-3);
\draw (0.5,-1.5)-- (1,-3);
\draw (-0.5,-1.5)-- (1,-3);
\draw (2.5,-1.5)-- (1,-3);

\draw (0,0) node[anchor=east] {$x_6$};
\draw (-0.5,1) node[anchor=east] {$x_{8}$};
\draw (0,2) node[anchor=south] {$x_{10}$};
\draw (2,0) node[anchor=west] {$x_7$};
\draw (2,2) node[anchor=south] {$x_{11}$};
\draw (2.5,1) node[anchor=west] {$x_{9}$};
\draw (-0.5,-1.5) node[anchor=east] {$x_{4}$};
\draw (0.5,-1.5) node[anchor=west] {$x_{3}$};
\draw (1.5,-1.5) node[anchor=east] {$x_{1}$};
\draw (2.5,-1.5) node[anchor=west] {$x_{2}$};
\draw (1,-3) node[anchor=north] {$x_{5}$};

\begin{scriptsize}
\draw [fill=black] (0,0) circle (1.5pt);
\draw [fill=black] (-0.5,1) circle (1.5pt);
\draw [fill=black] (2,0) circle (1.5pt);
\draw [fill=black] (0,2) circle (1.5pt);
\draw [fill=black] (2,2) circle (1.5pt);
\draw [fill=black] (2.5,1) circle (1.5pt);
\draw [fill=black] (-0.5,-1.5) circle (1.5pt);
\draw [fill=black] (0.5,-1.5) circle (1.5pt);
\draw [fill=black] (1.5,-1.5) circle (1.5pt);
\draw [fill=black] (2.5,-1.5) circle (1.5pt);
\draw [fill=black] (1,-3) circle (1.5pt);
\end{scriptsize}
\end{tikzpicture}
\caption{ 
The first graph $G$ on 11 vertices and 25 edges
with $\deg(h_{R/I(G)}(t)) = 2$
and $\v(I(G))=3$.}
    \label{fig:11vertex}
\end{figure}
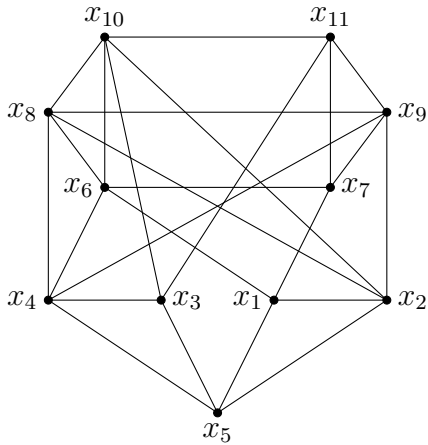

Before pressing forward, it is natural to ask if there is a ``smaller"
base case available. Amazingly, an extensive computer search has shown
that the example of \Cref{fig:11vertex} {\it is} one of the 
two smallest examples. We used {\it Macaulay2} to check
all graphs up to 11 vertices and 25 edges to produce the next statement.
Our computation involved checking well over 100,000,000 graphs (there are
86,318,670 non-isomorphic graphs on 11 vertices and 25 edges alone).

\begin{theorem}
    If $G$ is a graph with $|V(G)| \leq 10$, or
    $|V(G)| = 11$ and $|E(G)| \leq 24$, then $\v(I(G)) \leq \deg (h_{R/I(G)}(t))$.  If $G$ is a graph
    with $|V(G)|=11$ and $|E(G)|=25$, then $\v(I(G)) \leq \deg (h_{R/I(G)}(t))$,
    {\bf except} if $G$ is the graph of \Cref{fig:11vertex} 
    or \Cref{fig:secondminex}.  In both
    these two exceptional cases, $\deg(h_{R/I(G)}(t)) = 2
    < 3 = \v(I(G))$.
\end{theorem}

\begin{figure}[ht]
    \centering
\begin{tikzpicture}[scale=0.9,line cap=round,line join=round,>=triangle 45,x=1.5cm,y=1.5cm]
\coordinate (1) at (-0.5,1);
\coordinate (2) at (-1.5,-0);
\coordinate (3) at (2,1);
\coordinate (4) at (2,-0.5);
\coordinate (5) at (0.5,-0);
\coordinate (6) at (-1.5,-1);
\coordinate (7) at (3,-0.5);
\coordinate (8) at (-0.5,2);
\coordinate (9) at (0.5,-1);
\coordinate (10) at (2,2);
\coordinate (11) at (1,-2.5);

\draw (1) -- (4);
\draw (1) -- (5);
\draw (2) -- (5);
\draw (2) -- (6);
\draw (3) -- (6);
\draw (3) -- (7);
\draw (4) -- (7);
\draw (1) -- (8);
\draw (2) -- (8);
\draw (4) -- (8);
\draw (6) -- (8);
\draw (1) -- (9);
\draw (3) -- (9);
\draw (5) -- (9);
\draw (6) -- (9);
\draw (2) -- (10);
\draw (3) -- (10);
\draw (5) -- (10);
\draw (7) -- (10);
\draw (2) -- (11);
\draw (4) -- (11);
\draw (5) -- (11);
\draw (6) -- (11);
\draw (7) -- (11);
\draw (9) -- (11);

\node at (1) [above right] {$x_1$};
\node at (2) [left] {$x_2$};
\node at (3) [above left] {$x_3$};
\node at (4) [left] {$x_4$};
\node at (5) [right] {$x_5$};
\node at (6) [left] {$x_6$};
\node at (7) [right] {$x_7$};
\node at (8) [above] {$x_8$};
\node at (9) [below left] {$x_9$};
\node at (10) [above] {$x_{10}$};
\node at (11) [below] {$x_{11}$};

\foreach \i in {1,...,11}
  \fill (\i) circle (2pt);

\end{tikzpicture}
\caption{The second graph $G$ on 11 vertices and 25
edges with $\deg(h_{R/I(G)}(t)) = 2$
and $\v(I(G))=3$.} \label{fig:secondminex}
\end{figure}
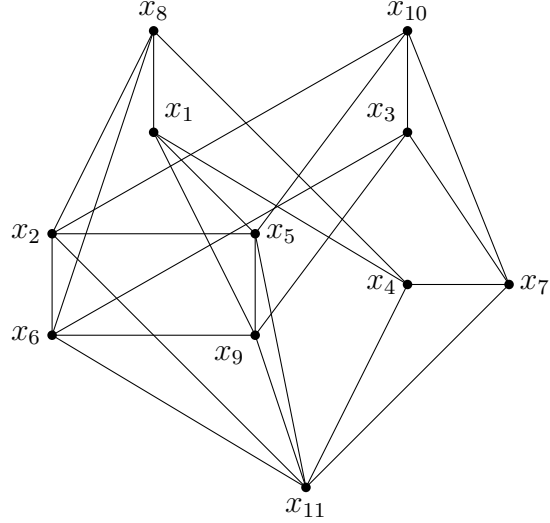
To demonstrate that the $\v$-number can be arbitrarily
larger than the degree of the $h$-polynomial, we describe a special construction of a graph built from a collection of graphs. We will then apply this construction to $n$ copies of the graph $G$ of \Cref{fig:11vertex}.

\noindent \textbf{Construction 1.} 
Let $G_1,\ldots,G_n$ be $n$ simple graphs. For each 
$1\leq i\leq n$, let $A_i\subseteq V(G_i)$ be such that
$\alpha(G_i)-\alpha(G_i\setminus A_i)=t_i\geq 1$ 
is odd and $\deg(G_i)-\deg(G_i\setminus A_i)=t_i-1$, 
where $\deg(G_i):=\deg(h_{R_i/I(G_i)}(t))$ with 
$R_i=\mathbb{K}[V(G_i)]$. Moreover, we assume that all the 
leading coefficients of the $h$-polynomials of 
$G_1,\ldots,G_n, G_1\setminus A_1,\ldots,G_n\setminus A_n$ are of the same sign, i.e., either all positive or all negative. Using the graphs $G_1,\ldots,G_n$ and the subsets of vertices $A_1,\ldots,A_n$, we construct a new graph $H_n$ as follows:
\begin{enumerate}
    \item[$\bullet$] $V(H_n)=V(G_1)\cup\cdots\cup V(G_n)\cup\{y_1,\ldots,y_n\}$.
    \item[$\bullet$] $E(H_n)=E(G_1)\cup\cdots\cup E(G_n)\cup\{\{y_{i},x\}\mid x\in A_i, i\in [n]\}\cup \{\{y_i,y_j\}\mid i\neq j\}$.
\end{enumerate}

\begin{lemma}\label{lem:deg}
    Let $H_n$ be the graph as in Construction 1 and $R=\mathbb{K}[V(H_n)]$. Then
    $$\dim(R/I(H_n))=\sum_{i=1}^{n}\alpha(G_i) \text{ and } \deg(h_{R/I(H_n)}(t))=1 + \sum_{i=1}^{n}\deg(h_{R_{i}/I(G_i)}(t)).$$
\end{lemma}
\begin{proof}
    For simplicity, we use the notation introduced in 
    Construction 1. We first prove the dimension formula. 
    
    Let $A$ be a maximal independent set of $H_n$ such that $\alpha(H_n)=\vert A\vert$, i.e., $\dim(R/I(H_n))=\vert A\vert$. Since there is no edge from $G_i$ to $G_j$ ($i\neq j$), a union of independent sets of $G_1,\ldots,G_n$ will also be an independent set in $H_n$. Thus, $\vert A\vert \geq \sum_{i=1}^{n}\alpha(G_i)$. 
    
    If $A\cap \{y_1,\ldots,y_n\}=\emptyset$, 
    then $\vert A\cap V(G_i)\vert\leq\alpha(G_i)$ for 
    each $1\leq i\leq n$. In this case, 
    $A\cap V(G_1),\ldots,A\cap V(G_n)$ forms
    a partition of $A$, and thus, 
    $\vert A\vert\leq\sum_{i=1}^{n}\alpha(G_i)$, thus
    giving the desired conclusion. On the
    other hand, suppose 
    $A\cap \{y_1,\ldots,y_n\}\neq \emptyset$. 
    Since the induced subgraph on $\{y_1,\ldots,y_n\}$ is
    a complete graph, and since $A$ is an independent set, we have $\vert A\cap\{y_1,\ldots,y_n\}\vert=1$. 
    Without loss of generality, we assume $y_1\in A$. 
    Since $\{y_1,x\}\in E(H_n)$ for every $x\in A_1$, 
    $A\cap V(G_1)=A\cap V(G_1\setminus A_1)$. Note that 
    $\vert (A\cap V(G_1))\cup\{y_1\}\vert =\vert A\cap 
    V(G_1\setminus A_1)\vert+1\leq
    \alpha(G_1\setminus A_1)+1\leq \alpha(G_1)$ 
    by our construction. In this case, the sets 
    $(A\cap V(G_1))\cup\{y_1\}, A\cap V(G_2),\ldots,
    A\cap V(G_n)$ forms a partition of $A$. Therefore, 
    $\vert A\vert\leq \sum_{i=1}^{n}\alpha(G_i)$, which 
    gives the desired conclusion in this case.  The
    formula for the dimension thus holds.

    We now prove the formula for the degree of the
    $h$-polynomial.  Recall that we write
    $\deg(G)$ for $\deg(h_{R/I(G)}(t))$. We also write $c(G)$ for the leading coefficient of  $h_{R/I(G)}(t)$. We now prove by 
    induction that $\deg(H_n)=1+\sum_{i=1}^{n}\deg(G_i)$ and the sign of $c(H_n)$ is positive (or negative)  if $c(G_i)$ and $c(G_i\setminus A_i)$ are all positive (or negative) for all $1\leq i\leq n$.

    For $n=1$, consider the following short exact sequence:
    $$0\rightarrow (R/I(H_1):(y_1))(-1)
    \stackrel{\times y_1}{\longrightarrow} R/I(H_1)\rightarrow R/(I(H_1)+(y_1))\rightarrow 0.$$
    Note that $I(H_1):(y_1)=(A_1)+I(G_1\setminus A_1)$ and $I(H_1)+(y_1)=(y_1)+I(G_1)$. Therefore, by the additivity of the Hilbert series on the above short exact sequence, we get
     \begin{align*}
        H_{R/I(H_1)}(t)&=t\cdot H_{R/I(H_1):(y_1)}(t)+H_{R/(I(H_1)+(y_1))}(t)\\
        &=\frac{t\cdot h_{R/I(G_1\setminus A_1)}(t)}{(1-t)^{1+\alpha(G_1\setminus A_1)}}+\frac{h_{R/I(G_1)}(t)}{(1-t)^{\alpha(G_1)}}\\
        &=\frac{t\cdot (1-t)^{t_1-1}\cdot h_{R/I(G_1\setminus A_1)}(t)+h_{R/I(G_1)}(t)}{(1-t)^{\alpha(G_1)}}.
    \end{align*}
    For the last equality, we are using the
    fact that $1 + \alpha(G_1\setminus A_1) +(t_1-1) =\alpha(G_1)$
    from our construction. Now $\deg(G_1\setminus A_1) +
    t_1 = \deg(G_1)+1$, by our construction's assumption.
    Therefore, from the above equation, we have 
    $\deg(H_1)=\deg(G_1)+1$. Since $t_1-1$ is even, 
    the leading coefficient of $(1-t)^{t_1-1}$ is positive
    (in fact, it is $1$), and so we 
    have $c(H_1)=c(G_1\setminus A_1)$. Hence, 
    the base case of the induction is true. 
    
    Now, for any $n\geq 2$, we consider the following 
    short exact sequence:
    \begin{equation}\label{eq1}
        0\rightarrow (R/I(H_n):(y_n))(-1)
        \stackrel{\times y_n}{\longrightarrow} R/I(H_n)\rightarrow R/(I(H_n)+(y_n))\rightarrow 0.
    \end{equation}
    Note that $$I(H_n):(y_n)=(y_1,\ldots,y_{n-1})+
    (A_n)+I(G_n\setminus A_n)+I(G_1)+\cdots+I(G_{n-1}).$$ 
    Then we have 
    $$\dim(R/I(H_n):(y_n))=1+
    \alpha(G_n\setminus A_n)+\sum_{i=1}^{n-1}\alpha(G_i)=
    1-t_n+\sum_{i=1}^{n}\alpha(G_i).$$ 
    By \Cref{lem:reg-deg-v} (iii) and by the conditions 
    of the construction
    $$\deg(h_{R/I(H_n):(y_n)}(t))=
    1-t_n+\sum_{i=1}^{n}\deg(G_i).$$ 
    Now, consider the ideal 
    $I(H_n)+(y_n)=I(H_{n-1})+I(G_n)+(y_n)$. 
    Then, we have 
    $\mathrm{dim}(R/(I(H_n)+(y_n))) = \sum_{i=1}^{n}\alpha(G_i)$. 
    Also, by the induction hypothesis and \Cref{lem:reg-deg-v} (iii), it follows that 
    $$\deg(h_{R/(I(H_n)+(y_n))}(t))=1+\sum_{i=1}^{n}\deg(G_i).$$
    Using the additivity of the Hilbert series 
    on the short exact sequence \eqref{eq1}, we get
     \begin{align*}
        H_{R/I(H_n)}(t)&=t\cdot H_{R/I(H_n):(y_n)}(t)+H_{R/(I(H_n)+(y_n))}(t)\\
        &=\frac{t\cdot h_{R/I(H_n):(y_n)}(t)}{(1-t)^{1-t_n+\sum_{i=1}^{n}\alpha(G_i))}}+\frac{h_{R/(I(H_n)+(y_n))}(t)}{(1-t)^{\sum_{i=1}^{n}\alpha(G_i)}}\\
        &=\frac{t\cdot (1-t)^{t_n-1}\cdot h_{R/I(H_n):(y_n)}(t)+h_{R/(I(H_n)+(y_n))}(t)}{(1-t)^{\sum_{i=1}^{n}\alpha(G_i)}}.
    \end{align*}

    By the given condition on the sign of the leading 
    coefficient of the $h$-polynomials, we can verify that 
    the sign of the leading coefficients of both 
    $h_{R/I(H_n):(y_n)}(t)$ and $h_{R/(I(H_n)+(y_n))}(t)$ 
    are the same (either positive or negative). Also, it 
    is given that $t_n-1$ is even. Therefore, using the 
    above equation and the formulas of the $h$-polynomials 
    of $R/I(H_n):(y_n$) and $R/(I(H_n)+(y_n))$, we obtain 
    $\deg(H_n)=1+\sum_{i=1}^{n}\deg(G_i)$ (the fact that the leading coefficients have the same sign ensures that the top coefficient is not canceled, so that the degree of the $h$-polynomial can be computed from 
    the displayed equation above). The claim for 
    $c(H_n)$ follows immediately. This now completes
    the induction proof.
\end{proof}

\noindent 
\textbf{Construction 2.} Let $G_1,\ldots,G_n$ be $n$ simple non-empty graphs. For each $1\leq i\leq n$, let $A_i\subseteq V(G_i)$ be such that $\v(I(G_i))\geq 1+\v(I(G_i\setminus A_i))$. Using the graphs $G_1,\ldots,G_n$ and the subsets of vertices $A_1,\ldots,A_n$, we construct a new graph $H_n$ as follows:
\begin{enumerate}
    \item[$\bullet$] $V(H_n)=V(G_1)\cup\cdots\cup V(G_n)\cup\{y_1,\ldots,y_n\}$.
    \item[$\bullet$] $E(H_n)=E(G_1)\cup\cdots\cup E(G_n)\cup\{\{y_{i},x\}\mid x\in A_i, i\in [n]\}\cup \{\{y_i,y_j\}\mid i\neq j\}$.
\end{enumerate}

\begin{lemma}\label{lem:v-num}
    Let $H_n$ be the graph as in Construction 2. Then for $n\geq 2$
    $$\v(I(H_n))=\min_{1\leq i\leq n} \left\{1+\v(I(G_i\setminus A_i))+\sum_{j\in [n]\setminus \{i\}}\v(I(G_j))\right\}.$$
\end{lemma}

\begin{proof}
    Fix any $i\in [n]$. For each 
    $j\in [n]\setminus \{i\}$, choose an 
    independent set $B_j$ of $G_j$ such that 
    $\vert B_j\vert=\v(I(G_j))$ and $N_{G_j}(B_j)$ is a 
    minimal vertex cover of $G_j$, i.e.,  
    $I(G_j):(\mathbf{x}_{B_j}) = (N_{G_j}(B_j))$. 
    Again, choose an independent set $B_i$ of 
    $G_i\setminus A_i$ such that 
    $N_{G_i\setminus A_i}(B_i)$ is a minimal vertex cover 
    of $G_i\setminus A_i$ with $\vert B_i\vert=\v(I(G_i\setminus A_i))$. 
    
    We now take $B=\{y_i\}\cup B_1\cup\cdots\cup B_n$. 
    It follows from Construction 2  that in the graph $H_n$
    the set $B$ is an independent set of $H_n$, and 
    $N_{H_n}(B_n)$ is a minimal vertex cover of $H_n$. 
    In particular, $I(H_n):(\mathbf{x}_{B})=(N_{H_n}(B))$, and thus, we have 
    $$\v(I(H_n))\leq 1+\v(I(G_i\setminus A_i))+\sum_{j\in [n]\setminus \{i\}}\v(I(G_j)).$$
    Since $i$ was chosen arbitrarily from $[n]$, it follows that 
    $$\v(I(H_n))\leq \min_{1\leq i\leq n} \left\{1+\v(I(G_i\setminus A_i))+\sum_{j\in [n]\setminus \{i\}}\v(I(G_j))\right\}.$$
    
    For the reverse inequality, suppose $B'$ is an 
    independent set of $H_n$ such that $N_{H_n}(B')$ is a minimal vertex cover of $H_n$ and 
    $\vert B'\vert=\v(I(H_n))$. 
    If $B'\cap \{y_1,\ldots,y_n\}=\emptyset$, then 
    $\vert B'\cap V(G_i)\vert\geq \v(I(G_i))$ for each 
    $1\leq i\leq n$ because 
    $B'\cap V(G_i)$ is an independent 
    set of $G_i$ and $N_{G_i}(B'\cap V(G_i))$ is a vertex 
    cover of $G_i$. Since it is given that 
    $\v(I(G_i))\geq 1+\v(I(G_i\setminus A_i))$ for each 
    $i\in [n]$, we have $$\vert B'\vert \geq \min_{1\leq i\leq n} \left\{1+\v(I(G_i\setminus A_i))+\sum_{j\in [n]\setminus \{i\}}\v(I(G_j))\right\},$$
    thus finishing the proof in this case. 
    
    Now, let us assume 
    $A\cap \{y_1,\ldots,y_n\}\neq \emptyset$. 
    Then $\vert B'\cap \{y_1,\ldots,y_n\}\vert=1$ because 
    the induced graph
    on $\{y_1,\ldots,y_n\}$ is a complete graph and 
    $B'$ is an independent set. Without loss of 
    generality, we may assume $B'\cap\{y_1,\ldots,y_n\}=\{y_1\}$. 
    Now, consider the ideal $I(H_n):(y_1)=(A_1)+ (y_2,\ldots,y_n)+I(G_1\setminus A_1)+I(G_2)+\cdots+I(G_n)$.
    By Lemma \ref{lem:reg-deg-v} (iv) 
    we have $$\v(I(H_n):(y_1))=\v(I(G_1\setminus A_1))+\sum_{j=2}^{n}\v(I(G_j)).$$ Because 
    each $G_i$ is a non-empty graph, we have $\v(I(H_n))>1$. Thus, from the proof of \cite[Proposition 3.12(c)]{v-edge} (or \cite[Proposition 3.13(iii)]{ssvmon22}) it follows that $$\v(I(H_n))\geq 1+\v(I(H_n):(y_1))=1+\v(I(G_1\setminus A_1))+\sum_{j=2}^{n}\v(I(G_j)).$$ Consequently, we have 
    $$\v(I(H_n))\geq \min_{1\leq i\leq n} \left\{1+\v(I(G_i\setminus A_i))+\sum_{j\in [n]\setminus \{i\}}\v(I(G_j))\right\}.$$
    This completes the proof.
    \end{proof}

The previous lemmas now allow us to prove the first major result of this section.

\begin{theorem}\label{Thm:d<v}
    For every positive integer $n$, there exists a connected graph $H$ such that 
    $$\mathrm{v}(I(H))-\mathrm{deg}(h_{R/I(H)}(t))=n.$$
\end{theorem}

\begin{proof}
    For every positive integer $n$, let us construct a new 
    graph $H_n$ using $n$ copies of the graph $G$ given 
    in \Cref{fig:11vertex} and a complete graph $K_n$ 
    on $n$ vertices. Let $G_{i}$ be the graph isomorphic 
    to $G$ with $V(G_i)=\{x^{(i)}_1,\ldots, x^{(i)}_{11}\}$
    and using the isomorphism $x^{(i)}_j$ goes to $x_j$. 
    Now we define $H_n$ as follows:
    \begin{enumerate}
    \item[$\bullet$] $V(H_n)=V(G_1)\cup\cdots\cup V(G_n)\cup\{y_1,\ldots,y_n\}$, and
    \item[$\bullet$] $E(H_n)=E(G_1)\cup\cdots\cup E(G_n)\cup\{\{y_{i},x^{(i)}_j\}\mid j\in[5], i\in [n]\}\cup \{\{y_i,y_j\}\mid i\neq j\}$.
\end{enumerate}
Observe that we are attaching each $y_i$ to the five
vertices $\{x_1^{(i)},x_2^{(i)},x_3^{(i)},x_4^{(i)},x_5^{(i)}\}$ 
in each copy of $G$ in \Cref{fig:11vertex}.

Note that for the edge ideal of a graph, the degree of 
the corresponding $h$-polynomial and the $\v$-number 
does not depend on the characteristic of the base field.
If we let $A = \{x_1,\ldots,x_5\}$,
then  one can check using {\it Macaulay2} that 
the reduced Hilbert series of $\mathbb{K}[V(G)]/I(G)$ and 
$\mathbb{K}[V(G\setminus A)]/I(G\setminus A)$
are given by
$$H_{\mathbb{K}[V(G)]/I(G)}(t) = \frac{1+8t+11t^2}{(1-t)^3}
~~\mbox{and}~~
H_{\mathbb{K}[V(G\setminus A)]/I(G\setminus A)}(t) 
=\frac{1+4t+t^2}{(1-t)^2}.$$
Thus  $\alpha(G)=3$, $\alpha(G\setminus A)=2$, 
$\deg(G)=2$, 
$\deg(G\setminus A)=2$,
and also the leading coefficients of
the $h$-polynomial of 
$\mathbb{K}[V(G)]/I(G)$ and $\mathbb{K}[V(G\setminus A)]/I(G\setminus A)$ have the same
sign. We also use {\it Macaulay2} to 
compute
$\v(I(G))=3$ 
and $\v(I(G\setminus A))=2$. 
Therefore, by considering 
$A_i = \{x^{(i)}_1,\ldots,x^{(i)}_5\}$, we have $\alpha(G_i)-\alpha(G_i\setminus A_i)=1$, $\deg(G_i)-\deg(G_i\setminus A_i)=0=1-1$, and $3=\v(I(G_i))\geq 1+\v(I(G_i\setminus A_i))=3$. 

Hence, the graph $H_n$ satisfies the conditions given in both Construction 1 and Construction 2. 
Thus, by \Cref{lem:deg} and \Cref{lem:v-num}, it follows that
$$\v(I(H_n))=3n \text{ and } \deg(H_n)=1+2n.$$
Thus for  any integer $n \geq 1$,
 $H=H_{n+1}$ satisfies $\mathrm{v}(I(H))-\mathrm{deg}(h_{R/I(H)}(t))=n$.
    \end{proof}

For the next main result, we need to know the Hilbert series of complete graphs and star graphs. A 
{\it star graph} on $n$ vertices, denoted by $K_{1,n-1}$, is a graph in which, after relabeling the vertices, we have $V(K_{1,n-1})=\{x_1,\ldots,x_n\}$ and $E(K_{1,n-1})=\{\{x_1,x_i\}\mid 2\leq i\leq n\}$. The Hilbert series of edge ideals for complete and star graphs are likely well-established or can be inferred as direct consequences of earlier works. However, due to the absence of a suitable reference, we include 
a short proof.

\begin{proposition}\label{prop:hilb-complete-star}
    The Hilbert series of the edge ideals of complete and star graphs are:
    \begin{enumerate}
    \item[(i)] $H_{R/I(K_n)}(t)=\frac{1+(n-1)t}{(1-t)}$, and
    \item[(ii)] $H_{R/I(K_{1,n-1})}(t)=\frac{1+t(1-t)^{n-2}}{(1-t)^{n-1}}$.
\end{enumerate}
\end{proposition}

\begin{proof}
    For a simple graph $G$, there is a nice combinatorial interpretation of $H_{R/I(G)}(t)$ in terms of the independent sets of $G$ due to Stanley \cite{stanley96}, which is given by
    \begin{equation}\label{eq:stanleyform}
H_{R/I(G)}(t)=\sum_{i=0}^{\alpha(G)}\frac{f_{i-1}t^{i}}{(1-t)^{i}},
\end{equation}
where $f_{i-1}$ denotes the number of independent sets of cardinality $i$ in $G$ (by convention $f_{-1} = 1$). For the complete graph $K_n$, we have $f_{-1}=1$, $f_0=n$, and $f_{i}=0$ for $i\geq 1$, and 
for $K_{1,n}$ we have $f_{-1}=1$, $f_0 = n$, and $f_{i} 
=\binom{n-1}{i+1}$ for $i=1,\ldots,n-2$. The formulas
in the statement are followed by using \eqref{eq:stanleyform}
and these specific values.
\end{proof}

We now prove our second main result that for any non-negative integer $m$, there exists a connected graph $G$ with $\deg(h_{R/I(G)}(t)) - \v(I(G))=m$. In fact, we prove the stronger result that for all $1 \leq v \leq d$, there is a connected graph $G$ with $(v,d) = (\v(I(G)),\deg(h_{R/I(G)}(t)))$. We 
require the following construction.

\noindent 
\textbf{Construction 3.}  For any integers $1 \leq v \leq d$, define the graph $G = H(v,d)$ to be the graph with the vertex set
$V(G) = \{x_i,y_i,w_i ~|~ 1 \leq i \leq v\} \cup
\{z_1,\ldots,z_{d-v}\}$
and edge set
\begin{eqnarray*}
E(G) &=& \{\{x_i,y_i\},\{y_i,w_i\},\{w_i,x_i\} ~|~ 1 \leq
i \leq v\} \cup \{\{x_1,x_i\} ~|~ 2 \leq i \leq v\} \\
& & \cup ~\{\{w_1,z_j\} ~|~ 1 \leq j \leq d-v\}.
\end{eqnarray*}
In other words, $G$ consists of $v$ triangles,
a $K_{1,v-1}$ on the vertices $\{x_1,\ldots,x_v\}$,
and  $d-v$ leaves attached to the vertex $w_1$.
The graph $H(4,7)$ is given in Figure \ref{fig(v,d)}.

\begin{figure}[h]
    \centering
\begin{tikzpicture}[line cap=round,line join=round,>=triangle 45,x=1.5cm,y=1cm]

\draw (-1,0)-- (0,0);
\draw (-0.5,1)-- (-1,0);
\draw (-0.5,1)-- (0,0);

\draw (.5,0)-- (1.5,0);
\draw (1,1)-- (.5,0);
\draw (1,1)-- (1.5,0);

\draw (-0.5,1)-- (1.5,2);
\draw (1,1) -- (1.5,2);
\draw (2.5,1)-- (1.5,2);
\draw (2,0)-- (3,0);
\draw (2,0)-- (2.5,1);
\draw (3,0)-- (2.5,1);
\draw (1.5,2)-- (1,3);
\draw (1.5,2)-- (2,3);
\draw (1,3)-- (2,3);
\draw (2,3)-- (3,3.5);
\draw (2,3)-- (3,3);
\draw (2,3)-- (3,2.5);

\draw (1.5,2) node[anchor=west] {$x_1$};
\draw (-0.5,1) node[anchor=south east] {$x_{2}$};
\draw (1,1) node[anchor=south east] {$x_{3}$};
\draw (2.5,1) node[anchor=south west] {$x_{4}$};

\draw (-1,0) node[anchor=north] {$y_2$};
\draw (0,0) node[anchor=north] {$w_2$};
\draw (.5,0) node[anchor=north] {$y_3$};
\draw (1.5,0) node[anchor=north] {$w_3$};
\draw (2,0) node[anchor=north] {$y_4$};
\draw (3,0) node[anchor=north] {$w_4$};
\draw (1,3) node[anchor=south] {$y_1$};
\draw (2,3) node[anchor=south] {$w_1$};
\draw (3,3.5) node[anchor=west] {$z_1$};
\draw (3,3) node[anchor=west] {$z_2$};
\draw (3,2.5) node[anchor=west] {$z_{3}$};

\begin{scriptsize}
\draw [fill=black] (-1,0) circle (1.5pt);
\draw [fill=black] (0,0) circle (1.5pt);
\draw [fill=black] (.5,0) circle (1.5pt);
\draw [fill=black] (1.5,0) circle (1.5pt);
\draw [fill=black] (2,0) circle (1.5pt);
\draw [fill=black] (3,0) circle (1.5pt);
\draw [fill=black] (-0.5,1) circle (1.5pt);
\draw [fill=black] (1,1) circle (1.5pt);
\draw [fill=black] (2.5,1) circle (1.5pt);
\draw [fill=black] (1.5,2) circle (1.5pt);
\draw [fill=black] (1,3) circle (1.5pt);
\draw [fill=black] (2,3) circle (1.5pt);
\draw [fill=black] (3,2.5) circle (1.5pt);
\draw [fill=black] (3,3) circle (1.5pt);
\draw [fill=black] (3,3.5) circle (1.5pt);
\end{scriptsize}
\end{tikzpicture}
\caption{The graph $H(4,7)$}
    \label{fig(v,d)}
\end{figure}
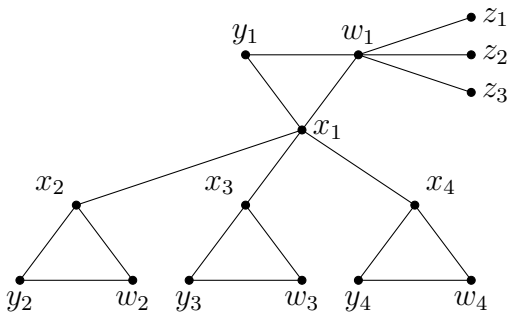

\begin{theorem}\label{thm:v<=d}
    For all integers $1 \leq v \leq d$, there exists a connected graph $G$
    with $(v,d) = (\v(I(G)),\deg(h_{R/I(G)}(t)))$.
\end{theorem}

\begin{proof}
Let $G = H(v,d)$ be the graph of Construction 3.  We 
will show that this graph satisfies $(v,d) = (\v(I(G)),\deg(h_{R/I(G)}(t)))$. 

We first show that $\v(I(G))=v$. Consider the set of 
vertices $A=\{y_1,x_2,\ldots,x_v\}$. 
Then from the construction of $G$, one can see
that $A$ is an independent set of $G$ and 
$N_{G}(A)$ is a minimal vertex cover of $G$. Thus, 
$I(G):(\mathbf{x}_{A})=(N_G(A))$.  Since 
$\vert A\vert =v$, it follows that $\v(I(G))\leq v$.   

Now let $B$ be an independent set of $G$ such that 
$N_G(B)$ is a minimal vertex cover of $G$, i.e., 
$I(G):(\mathbf{x}_{B})=(N_G(B))$. 
Since $N_G(B)$ is a minimal vertex cover of $G$, 
it must contain at most two elements of $\{x_i,y_i,w_i\}$
for each $1 \leq i \leq v$.  Indeed, if $N_G(B)$ contained
all three vertices $\{x_i,y_i,w_i\}$ 
for some $i$, we can remove $y_i$ and still have a
vertex cover, contradicting the minimality of $N_G(B)$.
Consequently, $\{x_i,y_i,w_i\}\cap B\neq \emptyset$ 
for each $1\leq i\leq v$. 
Because these $v$ sets are all disjoint, 
it follows that $|B| \geq v$.  
Since $B$ was chosen arbitrarily from 
$\mathcal{A}_{G}$, we have $\v(I(G))\geq v$, 
and consequently, $\v(I(G))=v$.

We now verify that $\deg(h_{R/I(G)}(t))=d$. 
Consider the following short exact sequence:
$$0\longrightarrow (R/I(G):(x_1))(-1) 
\stackrel{\times x_1}{\longrightarrow} 
R/I(G)\longrightarrow R/(I(G)+(x_1))\longrightarrow 0.$$
By the additivity of the Hilbert series on the short exact sequence, we get 
    \begin{align}\label{eq2}
        H_{R/I(G)}(t)=t\cdot H_{R/I(G):(x_1)}(t)+H_{R/(I(G)+(x_1))}(t).
    \end{align}
    
    The ideal $I(G):(x_1)=(x_2,\ldots,x_v,y_1,w_1)+I(G')$, 
    where $G'$ is a disjoint union of the graphs 
    $H_1,\ldots, H_v$ where $H_1$ is the graph of
    the isolated vertices $\{z_1,\ldots,z_{d-v}\}$ and 
    $H_i$ is the graph of the edge $\{y_i,w_i\}$ 
    for $2 \leq i \leq v$. In other words, $H_i 
    \cong K_2$ for $2 \leq i \leq v$.
    Thus, $\dim(R/I(G):(x_1))=1+(d-v)+(v-1)=d$, and by \Cref{lem:reg-deg-v} (ii) and \Cref{prop:hilb-complete-star}, we have
    \begin{align*}
        H_{R/I(G):(x_1)}(t)&=\frac{1}{(1-t)^{d-v+1}}\cdot \underbrace{\frac{1+t}{1-t}\,\,\cdots\,\,  \frac{1+t}{1-t}}_{(v-1) \text{ times}}=\frac{(1+t)^{v-1}}{(1-t)^{d}}.
    \end{align*}
    From the above expression, we see that $\deg(h_{R/I(G):(x_1)}(t))=v-1$ and the leading coefficient of $h_{R/I(G):(x_1)}(t)$ is $1$. 
    
    The ideal $I(G)+(x_1)= 
    I(G'') + (x_1)$, 
    where $G''$ is a disjoint union of the graphs $L_1,\ldots, L_v$ 
    where $L_1$ is the induced graph on $\{y_1,w_1,
    z_1,\ldots,z_{d-v}\}$ and $L_i$ the induced
    graph on $\{x_i,y_i,w_i\}$ for $2 \leq i \leq v$.
    Note that $L_1 \cong K_{1,d-v+1}$ and $L_i
    \cong K_3$ for $2 \leq i \leq v$.
    Thus, $\dim(R/I(G)+(x_1))=(d-v+1)+(v-1)=d$. Again, by \Cref{lem:reg-deg-v} (ii) and \Cref{prop:hilb-complete-star}, 
    \begin{align*}
        H_{R/(I(G)+(x_1))}(t)&=\frac{t(1-t)^{d-v}+1}{(1-t)^{d-v+1}}\cdot \underbrace{\frac{1+2t}{1-t}\,\,\cdots\,\,  \frac{1+2t}{1-t}}_{(v-1) \text{ times}}\\
        &=\frac{t(1+2t)^{v-1}(1-t)^{d-v}+(1+2t)^{v-1}}{(1-t)^{d}}.
    \end{align*}
From the above expression, we see that 
$\deg(h_{R/(I(G)+(x_1))}(t))=d$ as $v-1<d$ and the 
leading coefficient of $h_{R/(I(G)+(x_1))}(t)$ is $(-1)^{d-v}\cdot 2^{v-1}$. 

Now, for $v<d$, it follows directly from \eqref{eq2} that $\deg(h_{R/I(G)}(t))=d$. Suppose $v=d$.  In this
case both polynomials that appear in the numerator
of \eqref{eq2} have degree $d$.
By comparing the leading coefficients of $t\cdot h_{R/I(G):(x_1)}(t)$ and $h_{R/I(G):(x_1)}(t)$, the top coefficients
do not cancel, and thus \eqref{eq2} implies 
that $\deg(h_{R/I(G)}(t))=d$.
\end{proof}


\section{The sum of the v-number and degree}\label{sec4}

In the study of the comparison between the regularity and the degree of $h$-polynomials of edge ideals, it was shown in 
\cite[Theorem 13]{hmv19} that for a graph $G$ with $n$ 
vertices, $$\reg(R/I(G))+\deg(h_{R/I(G)}(t))\leq n.$$
Subsequently,  \cite[Theorem 6.4]{bkoss24} classified all 
those graphs for which the above inequality becomes an 
equality. In particular, it was shown that for a connected 
graph $G$ with $n$ vertices, 
$\reg(R/I(G))+\deg(h_{R/I(G)}(t))= n$ if and only if $G$ is a Cameron-Walker graph with no pendant triangles.  
Inspired by this work, in this section, we derive similar properties about
$\v(I(G)) + \deg(h_{R/I(G)}(t))$.

Our first step is to prove a relation 
between the $\v$-number and $\beta(G)$, the vertex covering number 
of $G$. 

\begin{theorem}\label{thm:v<vcnumber}
    Let $G$ be a graph. 
    Then $\v(I(G)) \leq \beta(G)$. Furthermore, 
    $\v(I(G))=\beta(G)$ if and only if $G$ is a disjoint union of star graphs.
\end{theorem}
\begin{proof}
For the first statement, see \cite[Proposition 3.14]{ssvmon22}.

Since $\v(I(G))$ and $\beta(G)$ are additive on a disjoint 
union of graphs, it is enough to consider $G$ to be 
connected. Note that $G$ is a star graph if and only if 
$\beta(G)=1$, and thus, for a star graph $G$, we have 
$\v(I(G)) = \beta(G)=1$. 
Therefore, if $\v(I(G))<\beta(G)$, then $G$ cannot be a star graph. 

Now, suppose $G$ is not a star graph. We then have 
$\beta(G)>1$. We proceed by induction on $\beta(G)$. 
The base case is $\beta(G)=2$. In this case, 
let us choose a minimal vertex cover $\{x_1,x_2\}$ of 
$G$. We consider the following two cases:

    \noindent
    \textbf{Case I.} Suppose $\{x_1,x_2\}\in E(G)$. 
    Then we observe that $N_{G}(x_1)$ and 
    $N_{G}(x_2)$ are both minimal vertex covers of $G$. 
    Thus, $I(G):(x_1)=(N_{G}(x_1))$ and 
    $I(G):(x_2)=(N_{G}(x_2))$, which implies that $\v(I(G))=1 < 2 = \beta(G)$.
    
    \noindent 
    \textbf{Case II.} Suppose $\{x_1,x_2\}\not\in E(G)$. 
    Since $G$ is connected and $\{x_1,x_2\}$ is a minimal 
    vertex cover of $G$, we must have 
    $N_{G}(x_1)\cap N_{G}(x_2)\neq \emptyset$. 
    Pick a vertex $z\in N_{G}(x_1)\cap N_{G}(x_2)$. 
    Then $N_{G}(z)=\{x_1,x_2\}$ as $\{x_1,x_2\}$ 
    is a vertex cover of $G$. In this case, 
    $I(G):(z)=(x_1,x_2)$, which gives $\v(I(G))=1<2=\beta(G)$.

    For the induction step, we assume $\beta(G)>2$. Choose 
    a minimal vertex cover $C$ of $G$ with 
    $\vert C\vert =\beta(G)$. Consider any vertex 
    $x\in C$ and the graph $G'=G\setminus N_{G}[x]$. 
    Then, we have $\beta(G')<\beta(G)$. Again, 
    we consider two cases as follows:

    \noindent
    \textbf{Case A.} Suppose $\beta(G')=\beta(G)-1$. First, assume $G'$ is a disjoint union of $k$ star graphs and some isolated vertices. Then $\beta(G')=k$, and so, $\beta(G)=k+1$. Let $y_1,\ldots,y_k$ be the center vertex of these $k$ star graphs. Since $G$ is connected, each $y_i$ is adjacent to at least one neighbor of $x$ in $G$. Then, we can choose $m$ neighbour vertices $A=\{x_1,\ldots,x_m\}$ of $x$ with $m\leq k$ such that $\{y_1,\ldots,y_k\}\subseteq N_{G}(A)$. Again, since $\beta(G')=\beta(G)-1$, no two neighbours of $x$ in $G$ can be adjacent. In this case, $A$ is an independent set, and $I(G):(\x_{A})$ is a minimal prime of $I(G)$. Indeed $N_G(A)=\{x,y_1,\ldots,y_k\}$ as $\beta(G)=k+1$. Consequently, we have $\v(I(G))\leq \vert A\vert= m\leq k<k+1=\beta(G)$. 
    
    Now, let us assume that $G'$ is not a disjoint union of star graphs and isolated vertices.
    By \cite[Proposition 3.12(a)]{v-edge} 
    (or, \cite[Proposition 3.13(i)]{ssvmon22}), 
    we have $\v(I(G))\leq \v(I(G):(x))+1$. Since 
    $I(G):(x)=I(G')+(N_{G}(x))$, it follows that 
    $\v(I(G):(x))=\v(I(G'))$. Now, $\beta(G)>2$ 
    implies $\beta(G')=\beta(G)-1\geq 2$. 
    Therefore, by the induction hypothesis and the above inequality, we get 
    $$\v(I(G))\leq \v(I(G'))+1<\beta(G')+1=\beta(G).$$
    
    \noindent 
    \textbf{Case B.} Suppose $\beta(G')<\beta(G)-1$. Then, only using \cite[Proposition 3.12(a)]{v-edge} and \cite[Proposition 3.14]{ssvmon22}, we have $\v(I(G))\leq \v(I(G'))+1\leq \beta(G')+1<\beta(G)$.
\end{proof}

We now come to the main result of this section, which presents a $\v$-number analog of the aforementioned results of \cite{bkoss24,hmv19}.

\begin{theorem}\label{thm:sumvdeg}
    If $G$ is a simple graph on $n$ vertices, then,
    $$\v(I(G))+\deg(h_{R/I(G)}(t))\leq n.$$
    Moreover, $\v(I(G))+\deg(h_{R/I(G)}(t))= n$ if 
    and only if $G$ is a disjoint union of star graphs.
\end{theorem}

\begin{proof}
    Recall from \eqref{eq:stanleyform},
    that the Hilbert series of $R/I(G)$
    can be 
    expressed as
$$H_{R/I(G)}(t)=\sum_{i=0}^{\alpha(G)}\frac{f_{i-1}t^{i}}{(1-t)^{i}},$$
where $f_{i-1}$ denotes the number of independent sets of cardinality $i$ in $G$. It follows from 
\eqref{eq:stanleyform} that $\deg(h_{R/I(G)}(t))\leq \alpha(G)$. 
Since $\v(I(G)) \leq \beta(G)$ by \Cref{thm:v<vcnumber}, we have 
$$\v(I(G))+\deg(h_{R/I(G)}(t))\leq \beta(G) + \alpha(G) = n,$$
where the last equality is well-known (it follows from the fact that the complement of a minimal vertex cover is a maximal independent set, and vice versa). 

For the second statement, we can reduce to the case that $G$ 
is connected. From the above inequalities, we observe that $\v(I(G))+\deg(h_{R/I(G)}(t))= n$ if and only if $\v(I(G))=\beta(G)$ and $\deg(h_{R/I(G)}(t))=\alpha(G)$. 
But by \Cref{thm:v<vcnumber}, $\v(I(G))=\beta(G)$ if and only if $G$ is a star graph. Thus, $\v(I(G))+\deg(h_{R/I(G)}(t))= n$ implies $G$ is a star graph.  For the reverse implication, if
$G$ is a star graph, then
Proposition \ref{prop:hilb-complete-star} gives
$\deg(h_{R/I(G)}(t))  = n-1$.  Since $1 \leq \v(I(G))$, the first
part of the proof then implies that $1+(n-1) \leq 
\v(I(G))+\deg(h_{R/I(G)}(t)) \leq n$, giving the desired result.
\end{proof}

\section{A menagerie of examples: comparison of v-number, degree, and regularity}\label{sec5}

In this section, we compare the $\v$-number, 
the degree of the $h$-polynomial, and the regularity of 
edge ideals of connected graphs. In particular, we show that all possible comparisons among these three invariants can occur.\par 

We first recall some useful notation.
If $M$ is a set of pairwise disjoint edges of $G$, then 
$M$ is called a \textit{matching} of $G$. 
An \textit{induced matching} of the graph ${G}$ is a 
matching $M=\{e_1,\ldots,e_m\}$ of ${G}$ such that the only edges of ${G}$ in the induced subgraph on $\bigcup_{i=1}^{m}{e_i}$ are $e_1,\ldots,e_m$. The \textit{induced matching number} of ${G}$ is the number of edges in a maximum induced matching of $G$ and is denoted by $\nu({G})$. It was proved in \cite[Corollary 6.9]{havan08} that for a 
chordal graph $G$, $\reg(R/I(G))=\nu(G)$. This fact will be 
used in this section.

To simplify our notation, we write $v,d$, and $r$ to 
denote $\v(I(G))$, $\deg(h_{R/I(G)}(t))$, and $\reg(R/I(G))$, respectively for a graph $G$, where $R=\K[V(G)]$.

\noindent
\begin{minipage}{0.6\textwidth} 
\begin{example}[{$v=d=r$}]\label{ex:5.1} {\rm 
Let $G$ be the graph in \Cref{fig(v=r=d)}, i.e., a single
edge. A straightforward calculation shows $v=d=1$. Also, $G$ being chordal with $\nu(G)=1$, implies $r=1$ for any field $\mathbb{K}$. Therefore, in this case, we have $v=d=r$.
}
\end{example}
\end{minipage}
\begin{minipage}{0.4\textwidth}
\begin{center}
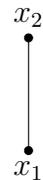

\begin{tikzpicture}[line cap=round,line join=round,>=triangle 45,x=1.5cm,y=1cm]

\draw (0,0)-- (0,1.5);

\draw (0,0) node[anchor=north] {$x_1$};
\draw (0,1.5) node[anchor=south] {$x_2$};

\begin{scriptsize}
\draw [fill=black] (0,0) circle (1.5pt);
\draw [fill=black] (0,1.5) circle (1.5pt);
\end{scriptsize}
\end{tikzpicture}
\captionof{figure}{A graph with $v=d=r$}
\label{fig(v=r=d)}
\end{center}
\end{minipage}

\noindent
\begin{minipage}{0.6\linewidth} 
\begin{example}[{$v=r<d$}]\label{ex:5.2}{\rm 
Let $G = K_{1,3}$ be the star graph in \Cref{fig(v=r<d)}.
Then $I(G):(x_1)=(x_2,x_3,x_4)$, which gives $v=1$. Since $G$ is chordal with $\nu(G)=1$, we have $r=1$ for any field $\mathbb{K}$. By Proposition \ref{prop:hilb-complete-star},
$d=3$. Thus, in this case, we have $v=r<d$.
}
\end{example}
\end{minipage}
\begin{minipage}{0.4\linewidth} 
    \centering
\begin{tikzpicture}[line cap=round,line join=round,>=triangle 45,x=1.5cm,y=1cm]

\draw (0,0)-- (0,1.5);
\draw (-1,0)-- (0,1.5);
\draw (1,0)-- (0,1.5);

\draw (0,0) node[anchor=north] {$x_3$};
\draw (-1,0) node[anchor=north] {$x_{2}$};
\draw (1,0) node[anchor=north] {$x_{4}$};
\draw (0,1.5) node[anchor=south] {$x_1$};

\begin{scriptsize}
\draw [fill=black] (0,0) circle (1.5pt);
\draw [fill=black] (1,0) circle (1.5pt);
\draw [fill=black] (-1,0) circle (1.5pt);
\draw [fill=black] (0,1.5) circle (1.5pt);
\end{scriptsize}
\end{tikzpicture}
\captionof{figure}{A graph with $v=r <d$}
    \label{fig(v=r<d)}
        \end{minipage}

\noindent
 \begin{minipage}{0.6\linewidth} 
\begin{example}[{$v=d<r$}]\label{ex:5.3}{\rm 
Let $G$ be the graph in \Cref{fig(v=d<r)}.
Then $N_{G}(x_{10})$ is a minimal vertex cover of $G$, and 
thus, $v=1$. Since $G$ is chordal with $\nu(G)=3$ (the edges
$\{x_1,x_2\}, \{x_4,x_5\}, \{x_6,x_7\}$ form an induced
matching), we have $r=3$ for any field $\mathbb{K}$. Using {\it Macaulay2}, we can check that $d=1$. Thus, in this case, we have $v=d<r$.  Our example
is similar to the graph found in \cite[Example 11]{hmv19} with $d<r$, but with additional edges adjacent to $x_6$.
}
\end{example}
\end{minipage}
\begin{minipage}{0.4\linewidth} 
    \centering
\begin{tikzpicture}[line cap=round,line join=round,>=triangle 45,x=1.5cm,y=1cm]

\draw (0,0)-- (0,1.5);
\draw (1,0.75)-- (0,0);
\draw (1,0.75)-- (0,1.5);
\draw (2,0)-- (2,1.5);
\draw (1,0.75)-- (2,0);
\draw (1,0.75)-- (2,1.5);
\draw (2.7,0)-- (2.7,1.5);
\draw (0.2,2.5)-- (1,3);
\draw (1.8,2.5)-- (1,3);
\draw (0.2,2.5)-- (1.8,2.5);
\draw (1,3)-- (0,0);
\draw (1,3)-- (0,1.5);
\draw (1,3)-- (1,0.75);
\draw (1,3)-- (2,0);
\draw (1,3)-- (2,1.5);
\draw (1,3)-- (2.7,0);
\draw (1,3)-- (2.7,1.5);
\draw (0.2,2.5)-- (0,0);
\draw (0.2,2.5)-- (0,1.5);
\draw (0.2,2.5)-- (1,0.75);
\draw (0.2,2.5)-- (2,0);
\draw (0.2,2.5)-- (2,1.5);
\draw (0.2,2.5)-- (2.7,0);
\draw (0.2,2.5)-- (2.7,1.5);

\draw (1.8,2.5)-- (0,0);
\draw (1.8,2.5)-- (0,1.5);
\draw (1.8,2.5)-- (1,0.75);
\draw (1.8,2.5)-- (2,0);
\draw (1.8,2.5)-- (2,1.5);
\draw (1.8,2.5)-- (2.7,0);
\draw (1.8,2.5)-- (2.7,1.5);

\draw (0,0) node[anchor=north] {$x_1$};
\draw (0,1.5) node[anchor=east] {$x_{2}$};
\draw (1,0.75) node[anchor=north] {$x_{3}$};
\draw (2,0) node[anchor=north] {$x_4$};
\draw (2,1.5) node[anchor=west] {$x_5$};
\draw (2.7,0) node[anchor=north] {$x_6$};
\draw (2.7,1.5) node[anchor=west] {$x_{7}$};
\draw (1,3) node[anchor=south] {$x_{10}$};
\draw (0.2,2.5) node[anchor=east] {$x_8$};
\draw (1.8,2.5) node[anchor=west] {$x_9$};

\begin{scriptsize}
\draw [fill=black] (0,0) circle (1.5pt);
\draw [fill=black] (0,1.5) circle (1.5pt);
\draw [fill=black] (2,0) circle (1.5pt);
\draw [fill=black] (2,1.5) circle (1.5pt);
\draw [fill=black] (1,0.75) circle (1.5pt);
\draw [fill=black] (2.7,0) circle (1.5pt);
\draw [fill=black] (2.7,1.5) circle (1.5pt);
\draw [fill=black] (1,3) circle (1.5pt);
\draw [fill=black] (0.2,2.5) circle (1.5pt);
\draw [fill=black] (1.8,2.5) circle (1.5pt);
\end{scriptsize}
\end{tikzpicture}
\captionof{figure}{A graph with $v=d < r$}
    \label{fig(v=d<r)}
\end{minipage}

\noindent
 \begin{minipage}{0.6\linewidth} 
\begin{example}[{$v<r<d$}]\label{ex:5.4}{\rm 
Let us consider $G$ to be the graph as shown in \Cref{fig(v<r<d)}; this is the path graph on five vertices. Then $N_{G}(x_{3})$ is a minimal vertex cover of $G$, and thus, $v=1$. Since $G$ is chordal with $\nu(G)=2$, we have $r=2$ for any field $\mathbb{K}$. Using {\it Macaulay2}, one can check that $d=3$. Hence, we have $v<r<d$ in this case.
}
\end{example}
\end{minipage}
\begin{minipage}{0.4\linewidth} 
    \centering
\begin{tikzpicture}[line cap=round,line join=round,>=triangle 45,x=1.5cm,y=1cm]

\draw (0,0)-- (0,1);
\draw (0,1)-- (1,2);
\draw (1,2)-- (2,1);
\draw (2,0)-- (2,1);

\draw (0,0) node[anchor=east] {$x_1$};
\draw (0,1) node[anchor=east] {$x_{2}$};
\draw (1,2) node[anchor=south] {$x_{3}$};
\draw (2,1) node[anchor=west] {$x_4$};
\draw (2,0) node[anchor=west] {$x_5$};

\begin{scriptsize}
\draw [fill=black] (0,0) circle (1.5pt);
\draw [fill=black] (0,1) circle (1.5pt);
\draw [fill=black] (1,2) circle (1.5pt);
\draw [fill=black] (2,1) circle (1.5pt);
\draw [fill=black] (2,0) circle (1.5pt);
\end{scriptsize}
\end{tikzpicture}
\captionof{figure}{A graph with $v < r < d$}
    \label{fig(v<r<d)}
\end{minipage}

\noindent
\begin{minipage}{0.6\linewidth} 
\begin{example}[{$v<d<r$}]\label{ex:5.5}{\rm 
Let us consider $G$ to be the graph as shown in \Cref{fig(v<d<r)}. Then $N_{G}(x_{8})$ is a minimal vertex cover of $G$, and thus, $v=1$. Since $G$ is chordal with $\nu(G)=3$, we have $r=3$ for any field $\mathbb{K}$. We can find $d=2$ by using {\it Macaulay2}. 
Hence, we have $v<d<r$ in this case.
}
\end{example}
\end{minipage}
\begin{minipage}{0.4\linewidth} 
    \centering
\begin{tikzpicture}[line cap=round,line join=round,>=triangle 45,x=1.5cm,y=1cm]

\draw (0,0)-- (0,1.5);
\draw (1,0.75)-- (0,0);
\draw (1,0.75)-- (0,1.5);
\draw (2,0)-- (2,1.5);
\draw (1,0.75)-- (2,0);
\draw (1,0.75)-- (2,1.5);
\draw (2.7,0)-- (2.7,1.5);
\draw (1,3)-- (0,0);
\draw (1,3)-- (0,1.5);
\draw (1,3)-- (1,0.75);
\draw (1,3)-- (2,0);
\draw (1,3)-- (2,1.5);
\draw (1,3)-- (2.7,0);
\draw (1,3)-- (2.7,1.5);

\draw (0,0) node[anchor=north] {$x_1$};
\draw (0,1.5) node[anchor=east] {$x_{2}$};
\draw (1,0.75) node[anchor=north] {$x_{3}$};
\draw (2,0) node[anchor=north] {$x_4$};
\draw (2,1.5) node[anchor=west] {$x_5$};
\draw (2.7,0) node[anchor=north] {$x_6$};
\draw (2.7,1.5) node[anchor=west] {$x_{7}$};
\draw (1,3) node[anchor=south] {$x_{8}$};

\begin{scriptsize}
\draw [fill=black] (0,0) circle (1.5pt);
\draw [fill=black] (0,1.5) circle (1.5pt);
\draw [fill=black] (2,0) circle (1.5pt);
\draw [fill=black] (2,1.5) circle (1.5pt);
\draw [fill=black] (1,0.75) circle (1.5pt);
\draw [fill=black] (2.7,0) circle (1.5pt);
\draw [fill=black] (2.7,1.5) circle (1.5pt);
\draw [fill=black] (1,3) circle (1.5pt);
\end{scriptsize}
\end{tikzpicture}
\captionof{figure}{A graph with $v < d < r$}
    \label{fig(v<d<r)}
\end{minipage}

\noindent
\begin{minipage}{0.6\linewidth} 
\begin{example}[{$v<d=r$}]\label{ex:5.6}{\rm 
Let $G$ be the graph as shown in \Cref{fig(v<d=r)}. Then $N_{G}(x_{6})$ is a minimal vertex cover of $G$, and thus, $v=1$. Since $G$ is chordal with $\nu(G)=2$, we have $r=2$ for any field $\mathbb{K}$. 
Finally, using {\it Macaulay2}, one can check that $d=2$. 
Hence, we have $v<d=r$.
}
\end{example}
\end{minipage}
\begin{minipage}{0.4\linewidth} 
    \centering
\begin{tikzpicture}[line cap=round,line join=round,>=triangle 45,x=1.5cm,y=1cm]

\draw (0,0)-- (0,1.5);
\draw (1,0.75)-- (0,0);
\draw (1,0.75)-- (0,1.5);
\draw (2,0)-- (2,1.5);
\draw (1,0.75)-- (2,0);
\draw (1,0.75)-- (2,1.5);
\draw (1,3)-- (0,0);
\draw (1,3)-- (0,1.5);
\draw (1,3)-- (1,0.75);
\draw (1,3)-- (2,0);
\draw (1,3)-- (2,1.5);

\draw (0,0) node[anchor=north] {$x_1$};
\draw (0,1.5) node[anchor=east] {$x_{2}$};
\draw (1,0.75) node[anchor=north] {$x_{3}$};
\draw (2,0) node[anchor=north] {$x_4$};
\draw (2,1.5) node[anchor=west] {$x_5$};
\draw (1,3) node[anchor=south] {$x_{6}$};

\begin{scriptsize}
\draw [fill=black] (0,0) circle (1.5pt);
\draw [fill=black] (0,1.5) circle (1.5pt);
\draw [fill=black] (2,0) circle (1.5pt);
\draw [fill=black] (2,1.5) circle (1.5pt);
\draw [fill=black] (1,0.75) circle (1.5pt);
\draw [fill=black] (1,3) circle (1.5pt);
\end{scriptsize}
\end{tikzpicture}

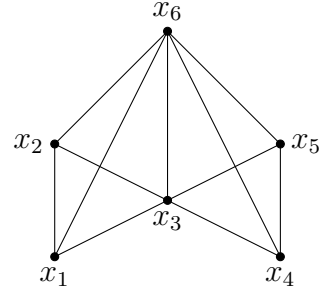
\captionof{figure}{A graph with $v < d = r$}
    \label{fig(v<d=r)}
\end{minipage}

In the next four examples, we make use of the graph of \Cref{fig:11vertex} to construct our examples.

\begin{example}[{$d=r<v$ and $d<v=r$}]\label{ex:5.7}{\rm
    Consider the graph $G$ as shown in \Cref{fig:11vertex}. Then, using {\it Macaulay2}, we get $d=2$ and $v=3$ as mentioned in the proof of \Cref{Thm:d<v}. Also, taking $\mathbb{K}=\mathbb{Q}$ and using {\it Macaulay2}, one can verify that $r=2$. Thus, in this case, we have $d=r<v$. Again, if we consider $\K=\ZZ/2\ZZ$, then we have $r=3$, and the $d$ and $v$ will be the same as they do not depend upon the characteristics of the base field. Thus, in this case, we have $d<v=r$.   This example originally came from
    work of Jaramillo and Villarreal \cite{v-edge}.
    }
\end{example}

\begin{example}[{$r<d=v$}]\label{ex:5.8}{\rm
    Consider the graph $G$ as shown in \Cref{fig:11vertex}. Now construct a graph $H$ such that 
    \begin{enumerate}
        \item[$\bullet$] $V(H)=V(G)\cup\{y\}$, and 
        \item[$\bullet$] $E(H)=E(G)\cup \{\{x_1,y\},\{x_2,y\},\{x_3,y\},\{x_4,y\}\}$.
    \end{enumerate}
   Then, using {\it Macaulay2} we find
   $\deg(h_{R/I(H)}(t))=\v(I(H))=3$ and $\reg(R/I(H))=2$, where $R=\mathbb{Q}[V(H)]$. Consequently, 
   we have $r<d=v$ in this case.
    }
\end{example}

\begin{example}[{$r<v<d$}]\label{ex:5.9}{\rm
    Again, let $G$ be as in  \Cref{fig:11vertex}
    and consider the graph $H$ defined as follows:
    \begin{enumerate}
        \item[$\bullet$] $V(H)=V(G)\cup \{y_1,y_2\}$, and 
        \item[$\bullet$] $E(H)=E(G)\cup \{\{x_i,y_1\},\{x_i,y_2\}\mid 1\leq i\leq 4\}$.
    \end{enumerate}
    Then we can use {\it Macaulay2} to find that 
    $$\reg(R/I(H))=2<\v(I(H))=3<\deg(h_{R/I(H)}(t))=4,$$ where $R=\mathbb{Q}[V(H)]$. Therefore, we have $r<v<d$ in this case.
    }
\end{example}

\begin{example}[{$r<d<v$ and $d<v<r$}]\label{ex:5.10}{\rm
     Let us construct a graph $H_n$ using $n$ copies of the graph $G$ given in Figure \ref{fig:11vertex}. Let $G_{i}$ be the graph isomorphic to $G$ with $V(G_i)=\{x^{(i)}_1,\ldots, x^{(i)}_{11}\}$ and the isomorphism $x^{(i)}_j$ goes to $x_j$. Now we define $H_n$ as follows:
     \begin{enumerate}
    \item[$\bullet$] $V(H_n)=V(G_1)\cup\cdots\cup V(G_n)\cup\{y_1,\ldots,y_n\}$,
    \item[$\bullet$] $E(H_n)=E(G_1)\cup\cdots\cup E(G_n)\cup\{\{y_{i},x^{(i)}_j\}\mid j\in [11], i\in [n]\}\cup \{\{y_i,y_j\}\mid i\neq j\}$.
\end{enumerate}
\noindent Recall that $[m] = \{1,\ldots,m\}$ for a positive integer $m$.

\noindent\textbf{Claim:} For every $n\geq 1$, we have $\deg(h_{R/I(H_n)}(t))=2n+1$, $\v(I(H_n))=3n-2$, and 
$$\reg(R/I(H_n))=
\begin{cases}
    2n \quad\text{ if } \mathbb{K=Q},\\
    3n \quad\text{ if } \mathbb{K=Z}/2\mathbb{Z}.
\end{cases}
$$ 
\textit{Proof of the claim.} 
For a square-free monomial ideal, 
the degree of the $h$-polynomial and the $\v$-number does not 
depend upon the characteristic of the base field. Thus, one can 
easily check using {\it Macaulay2} that $\alpha(G)=3$, 
$\alpha(G\setminus \{x_1,\ldots,x_{11}\})=0$, $\deg(G)=2$, 
$\deg(G\setminus \{x_1,\ldots,x_{11}\})=0$, $\v(I(G))=3$, 
$\v(I(G\setminus \{x_1,\ldots,x_{11}\}))=0$.
As shown in the proof of Theorem \ref{Thm:d<v}, the 
leading coefficient of $h_{R/I(G)}(t)$ is positive.  
Also, since $G \setminus \{x_1,\ldots,x_n\}$ is empty, 
the associated $h$-polynomial is $1$, whose leading coefficient is clearly positive. Therefore, by considering $A_i=\{x^{(i)}_1,\ldots,x^{(i)}_{11}\}$, we have $\alpha(G_i)-\alpha(G_i\setminus A_i)=3$, 
$\deg(G_i)-\deg(G_i\setminus A_i)=2=3-1$, and 
$3=\v(I(G_i))\geq 1+\v(I(G_i\setminus A_i))=1$. Hence, 
the graph $H_n$ can be viewed as a graph described in both 
Construction 1 and Construction 2. Thus, by \Cref{lem:deg} and 
\Cref{lem:v-num}, it follows that 
$\deg(h_{R/I(H_n)}(t))=1+n\cdot\deg(G)=2n+1$ and $\v(I(H_n))=1+0+3(n-1)=3n-2$.

Now, we prove $\reg(R/I(H_n))=2n$ if $\mathbb{K=Q}$. The base case $n=1$ can be verified directly using {\it Macaulay2}. Let us consider the 
following short exact sequence
$$0\rightarrow (R/I(H_n):(y_n))(-1)
\stackrel{\times y_n}{\longrightarrow} R/I(H_n)\rightarrow R/(I(H_n)+(y_n))\rightarrow 0.$$
    From the description of $I(H_n):(y_n)$ and $I(H_n)+(y_n)$ as given earlier in the proof of \Cref{lem:deg} and using the induction hypothesis, we have $\reg(R/I(H_n):(y_n))=2n-2$ and $\reg(R/(I(H_n)+(y_n)))=2n$. Since $\reg(R/I(H_n))\geq\reg(R/(I(H_n)+(y_n)))$, it follows that $\reg(R/I(H_n))= 2n$ by the regularity lemma
    (see \cite[Theorem 4.6]{CHHVT}). The case $\mathbb{K=Z}/2\mathbb{Z}$ is similar, but we make use of the fact that in the base case,
     $\reg(R/I(H_1)) = 3$.
}

Because of the claim, for $n \geq 4$, we have $r < d < v$ if 
$\mathbb{K} = \mathbb{Q}$, and $d < v < r$ if $\mathbb{K}=
\mathbb{Z}/2\mathbb{Z}$.
\end{example}

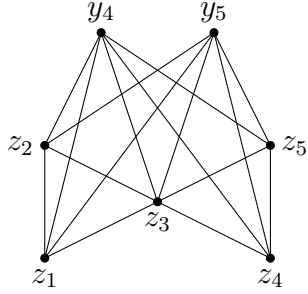
\begin{figure}[h]
    \centering
    \begin{tikzpicture}[line cap=round,line join=round,>=triangle 45,x=1.5cm,y=1cm]

\draw (0,0)-- (0,1.5);
\draw (1,0.75)-- (0,0);
\draw (1,0.75)-- (0,1.5);
\draw (2,0)-- (2,1.5);
\draw (1,0.75)-- (2,0);
\draw (1,0.75)-- (2,1.5);
\draw (0.5,3)-- (0,0);
\draw (0.5,3)-- (0,1.5);
\draw (0.5,3)-- (1,0.75);
\draw (0.5,3)-- (2,0);
\draw (0.5,3)-- (2,1.5);
\draw (1.5,3)-- (0,0);
\draw (1.5,3)-- (0,1.5);
\draw (1.5,3)-- (1,0.75);
\draw (1.5,3)-- (2,0);
\draw (1.5,3)-- (2,1.5);

\draw (0,0) node[anchor=north] {$z_1$};
\draw (0,1.5) node[anchor=east] {$z_{2}$};
\draw (1,0.75) node[anchor=north] {$z_{3}$};
\draw (2,0) node[anchor=north] {$z_4$};
\draw (2,1.5) node[anchor=west] {$z_5$};
\draw (0.5,3) node[anchor=south] {$y_{4}$};
\draw (1.5,3) node[anchor=south] {$y_{5}$};

\begin{scriptsize}
\draw [fill=black] (0,0) circle (1.5pt);
\draw [fill=black] (0,1.5) circle (1.5pt);
\draw [fill=black] (2,0) circle (1.5pt);
\draw [fill=black] (2,1.5) circle (1.5pt);
\draw [fill=black] (1,0.75) circle (1.5pt);
\draw [fill=black] (0.5,3) circle (1.5pt);
\draw [fill=black] (1.5,3) circle (1.5pt);
\end{scriptsize}
\end{tikzpicture}
    \caption{A graph $G$}
    \label{fig:z}
\end{figure}

\begin{example}[{$d<r<v$}]\label{ex:5.11} 
    Let us construct a graph $H$ using the graph $H_3$ as constructed in \Cref{Thm:d<v} and the graph $G$ of \Cref{fig:z} in the following way:
     \begin{enumerate}
    \item[$\bullet$] $V(H)=V(H_3)\cup V(G)$, and 

    \item[$\bullet$] $E(H)=E(H_3)\cup E(G)\cup E(K_5)$, 
\end{enumerate}
where $K_5$ is a complete graph on $\{y_1,\ldots,y_5\}$. Then, using {\it Macaulay2}, one can check that $d=8$. Also, we have $v=10$, and $r=9$ when we take $\mathbb{K=Q}$.
We were not able to verify these values directly using {\it Macaulay2}
since the computations would not finish. 
Instead, we give a short proof of this claim:

\noindent
\textbf{Claim:} $\v(I(H))=10$ and if we take $\mathbb{K=Q}$, then $\reg(R/I(H))=9$.

\noindent
\textit{Proof of the claim.} First, let us proof that $\v(I(H))=10$. Take an independent set $B$ of $H_3$ such that $I(H_3):(\mathbf{x}_{B})$ is a prime ideal and 
$\vert B\vert=\v(I(H_3))$. By \Cref{Thm:d<v}, it follows that $\vert B\vert=9$. Then, from the structure of $H$, we have $B\cup \{z_3\}$ is an independent set of $H$, whose neighbours form a minimal vertex cover of $H$. Consequently, $\v(I(H))\leq \vert B\cup\{z_3\}\vert=10$. For reverse inequality, let $A$ be an independent set of $H$ such that $N_{H}(A)$ is a minimal vertex cover of $H$ and $\vert A\vert =\v(I(H))$. Then $A\cap V(H_3)$ is an independent set of $H_3$, whose neighbours also form a vertex cover of $H_3$. Since $\v(I(H_3))= 9$, we have $\vert A\cap V(I(H_3))\vert\geq 9$. Note that neighbours of $A\cap V(I(H_3))$ in $H$ can never cover any of the edges of $
H_{\{z_1,\ldots,z_5\}} = G_{\{z_1,\ldots,z_5\}}$. Thus, $A\setminus V(H_3)\neq \emptyset$. Hence, $\vert A\vert> 9$, and consequently, we have $\v(I(H))=10$. 

Let $R=\mathbb{Q}[V(H)]$. To prove $\reg(R/I(H))=9$, we first use {\it Macaulay2} to check that $\reg(R/I(H_2))=5$. Next, let us prove $\reg(R/I(H_3))=7$. By looking at the ideals $I(H_3):(y_3)$ and $I(H_3)+(y_3)$ as described in the proof of \Cref{Thm:d<v}, one can obtain via {\it Macaulay2}  that $\reg(R/I(H_3):(y_3))=6$ and $\reg(R/(I(H_3)+(y_3)))=7$. Thus, by \cite[Lemma 2.10]{dhs13}, we get $\reg(R/I(H_3))=7$. Now, consider the ideal $I(H):(y_5)=(y_1,\ldots,y_4,z_1,\ldots,z_5)+I(H')$, where $H'$ is a disjoint union of $3$ copies of the graph shown in \Cref{fig:11vertex}. Then, we have $\reg(R/I(H):(y_5))=6$. Again, $(I(H)+(y_5)):(y_4)=(y_1,\ldots,y_3,y_5,z_1,\ldots,z_5)+I(H')$, and thus, we also have $\reg(R/(I(H)+(y_5)):(y_4))=6$. Now, observe that $I(H)+(y_4,y_5)=(y_4,y_5)+I(H_3)+I(G\setminus \{y_4,y_5\})$. Since $G\setminus \{y_4,y_5\}$ is a chordal graph with induced matching number $2$, we get $\reg(R/I(H)+(y_4,y_5))=9$. Therefore, using \cite[Lemma 2.10]{dhs13} again and again, we have $\reg(R/I(H))=9$. Consequently, we get $d<r<v$ for the graph $H$ in the setup $\mathbb{K=Q}$.
\end{example}

\begin{remark}
    Many of our examples make use of the fact that
    the graph $G$ of Figure \ref{fig:11vertex} has a regularity
    that depends upon the characteristic. Civan 
    \cite[Section 4]{civan23} constructed connected graphs for which the $\v$-number can be made arbitrarily larger than the regularity, independent of the characteristic of the base field $\K$.
    In \Cref{ex:5.7} to \Cref{ex:5.11}, we used $\K = \mathbb{Q}$ and the graph $G$ from \Cref{fig:11vertex}. 
    If we wish to obtain a similar result independent of the choice of the base field, we could instead consider
    Civan's example for our ``base graph".
\end{remark}

\newpage

\section*{Appendix A: Tables comparing the v-number and degree} 
We computed and plotted the tuple $(\v(I(G)),\deg(h_{R/I(G)}(t)))$
for all connected graphs $G$ with $2 \leq |V(G)| \leq 10$ (see Figure \ref{fig:comparevandd}). Note that by Theorem
 \ref{thm:sumvdeg}, we have $\v(I(G))+\deg(h_{R/I(G)}(t) \leq n$.
 In fact, the plots seem to imply that  $2\v(I(G)) + \deg(h_{R/I(G)}(t)) \leq n+1.$
\begin{figure}[h]
\begin{tikzpicture}[scale=0.7]
\foreach \point in {(1/2,1/2)}{
    \fill \point circle (2pt);
}

\draw[thin,->] (5.5,0) -- (8.5,0) node[right] {$v=\v(I)$};
\draw[thin,->] (6,-0.5) -- (6,5.0) node[above] {$d = \deg$};

\foreach \x [count=\xi starting from 0] in {1,2,3,4,5,6,7,8,9}{
    \draw (5.9,\x/2) -- (6.1,\x/2);
    \ifodd\xi
        \node[anchor=east] at (0,\x/2) {$\x$};
    \fi
 \draw (6+1/2,1pt) -- (6+1/2,-1pt);
 \draw (6+2/2,1pt) -- (6+2/2,-1pt);
 \draw (6+3/2,1pt) -- (6+3/2,-1pt);
     \node[anchor=north] at (6+1/2,0) {$1$};
     \node[anchor=north] at (6+2/2,0) {$2$};
     \node[anchor=north] at (6+3/2,0) {$3$};
}

\foreach \point in {(6+1/2,1/2),(6+1/2,2/2)}{
    \fill \point circle (2pt);
}

\draw[thin,->] (-0.5,0) -- (2.5,0) node[right] {$v=\v(I)$};
\draw[thin,->] (0,-0.5) -- (0,5.0) node[above] {$d = \deg$};

\foreach \x [count=\xi starting from 0] in {1,2,3,4,5,6,7,8,9}{
        \draw (1pt,\x/2) -- (-1pt,\x/2);
    \ifodd\xi
        \node[anchor=east] at (0,\x/2) {$\x$};
    \fi
 \draw (1/2,1pt) -- (1/2,-1pt);
 \draw (2/2,1pt) -- (2/2,-1pt);
 \draw (3/2,1pt) -- (3/2,-1pt);
     \node[anchor=north] at (1/2,0) {$1$};
     \node[anchor=north] at (2/2,0) {$2$};
     \node[anchor=north] at (3/2,0) {$3$};
}

\draw[thin,->] (11.5,0) -- (14.5,0) node[right] {$v=\v(I)$};
\draw[thin,->] (12,-0.5) -- (12,5.0) node[above] {$d = \deg$};

\foreach \x [count=\xi starting from 0] in {1,2,3,4,5,6,7,8,9}{
    \draw (11.9,\x/2) -- (12.1,\x/2);
    \ifodd\xi
        \node[anchor=east] at (0,\x/2) {$\x$};
    \fi
\draw (12+1/2,1pt) -- (12+1/2,-1pt);
 \draw (12+2/2,1pt) -- (12+2/2,-1pt);
 \draw (12+3/2,1pt) -- (12+3/2,-1pt);
     \node[anchor=north] at (12+1/2,0) {$1$};
     \node[anchor=north] at (12+2/2,0) {$2$};
     \node[anchor=north] at (12+3/2,0) {$3$};
}

\
\foreach \point in {(12+1/2,1/2),(12+1/2,2/2),(12+1/2,3/2)}{
    \fill \point circle (2pt);
}

\node at (3,4) {$n=2$};
\node at (9,4) {$n=3$};
\node at (16,4) {$n=4$};
\end{tikzpicture}
\vspace{.5cm}


\begin{tikzpicture}[scale=0.7]
  \foreach \point in {(1/2,1/2),(1/2,2/2),(2/2,2/2),(1/2,3/2),(1/2,4/2)}{
    \fill \point circle (2pt);
}

\draw[thin,->] (5.5,0) -- (8.5,0) node[right] {$v=\v(I)$};
\draw[thin,->] (6,-0.5) -- (6,5.0) node[above] {$d = \deg$};

\foreach \x [count=\xi starting from 0] in {1,2,3,4,5,6,7,8,9}{
    \draw (5.9,\x/2) -- (6.1,\x/2);
    \ifodd\xi
        \node[anchor=east] at (0,\x/2) {$\x$};
    \fi
\draw (6+1/2,1pt) -- (6+1/2,-1pt);
 \draw (6+2/2,1pt) -- (6+2/2,-1pt);
 \draw (6+3/2,1pt) -- (6+3/2,-1pt);
     \node[anchor=north] at (6+1/2,0) {$1$};
     \node[anchor=north] at (6+2/2,0) {$2$};
     \node[anchor=north] at (6+3/2,0) {$3$};
}

\foreach \point in {(6+1/2,1/2),(6+1/2,2/2),(6+1/2,3/2),(6+1/2,4/2),(6+1/2,5/2),
  (6+2/2,2/2),(6+2/2,3/2)}{
    \fill \point circle (2pt);
}

\draw[thin,->] (-0.5,0) -- (2.5,0) node[right] {$v=\v(I)$};
\draw[thin,->] (0,-0.5) -- (0,5.0) node[above] {$d = \deg$};

\foreach \x [count=\xi starting from 0] in {1,2,3,4,5,6,7,8,9}{
        \draw (1pt,\x/2) -- (-1pt,\x/2);
    \ifodd\xi
        \node[anchor=east] at (0,\x/2) {$\x$};
    \fi
 \draw (1/2,1pt) -- (1/2,-1pt);
 \draw (2/2,1pt) -- (2/2,-1pt);
 \draw (3/2,1pt) -- (3/2,-1pt);
     \node[anchor=north] at (1/2,0) {$1$};
     \node[anchor=north] at (2/2,0) {$2$};
     \node[anchor=north] at (3/2,0) {$3$};
}

\draw[thin,->] (11.5,0) -- (14.5,0) node[right] {$v=\v(I)$};
\draw[thin,->] (12,-0.5) -- (12,5.0) node[above] {$d = \deg$};

\foreach \x [count=\xi starting from 0] in {1,2,3,4,5,6,7,8,9}{
    \draw (11.9,\x/2) -- (12.1,\x/2);
    \ifodd\xi
        \node[anchor=east] at (0,\x/2) {$\x$};
    \fi
\draw (12+1/2,1pt) -- (12+1/2,-1pt);
 \draw (12+2/2,1pt) -- (12+2/2,-1pt);
 \draw (12+3/2,1pt) -- (12+3/2,-1pt);
     \node[anchor=north] at (12+1/2,0) {$1$};
     \node[anchor=north] at (12+2/2,0) {$2$};
     \node[anchor=north] at (12+3/2,0) {$3$};
}

\
\foreach \point in {(12+1/2,1/2),(12+1/2,2/2),(12+1/2,3/2),(12+1/2,4/2),(12+1/2,5/2),(12+1/2,6/2),
  (12+2/2,2/2),(12+2/2,3/2),(12+2/2,4/2)}{
    \fill \point circle (2pt);
}

\node at (3,4) {$n=5$};
\node at (9,4) {$n=6$};
\node at (16,4) {$n=7$};
\end{tikzpicture}
\vspace{.5cm}


\begin{tikzpicture}[scale=0.7]

  \foreach \point in {
    (1/2,1/2),(1/2,2/2),(1/2,3/2),(1/2,4/2),(1/2,5/2),(1/2,6/2),
    (1/2,7/2),
    (2/2,2/2),(2/2,3/2),(2/2,4/2),(2/2,5/2),
    (3/2,3/2)}{
    \fill \point circle (2pt);
}

\draw[thin,->] (5.5,0) -- (8.5,0) node[right] {$v = \v(I)$};
\draw[thin,->] (6,-0.5) -- (6,5.0) node[above] {$d = \deg$};

\foreach \x [count=\xi starting from 0] in {1,2,3,4,5,6,7,8,9}{
   
    \draw (5.9,\x/2) -- (6.1,\x/2);
    \ifodd\xi
    \fi
\draw (6+1/2,1pt) -- (6+1/2,-1pt);
 \draw (6+2/2,1pt) -- (6+2/2,-1pt);
 \draw (6+3/2,1pt) -- (6+3/2,-1pt);
     \node[anchor=north] at (6+1/2,0) {$1$};
     \node[anchor=north] at (6+2/2,0) {$2$};
     \node[anchor=north] at (6+3/2,0) {$3$};
}

\draw[thin,->] (-0.5,0) -- (2.5,0) node[right] {$v=\v(I)$};
\draw[thin,->] (0,-0.5) -- (0,5.0) node[above] {$d = \deg$};

\foreach \x [count=\xi starting from 0] in {1,2,3,4,5,6,7,8,9}{
        \draw (1pt,\x/2) -- (-1pt,\x/2);
    \ifodd\xi
        \node[anchor=east] at (0,\x/2) {$\x$};
    \fi
 \draw (1/2,1pt) -- (1/2,-1pt);
 \draw (2/2,1pt) -- (2/2,-1pt);
 \draw (3/2,1pt) -- (3/2,-1pt);
     \node[anchor=north] at (1/2,0) {$1$};
     \node[anchor=north] at (2/2,0) {$2$};
     \node[anchor=north] at (3/2,0) {$3$};
}

\draw[thin,->] (11.5,0) -- (14.5,0) node[right] {$v=\v(I)$};
\draw[thin,->] (12,-0.5) -- (12,5.0) node[above] {$d = \deg$};

\foreach \x [count=\xi starting from 0] in {1,2,3,4,5,6,7,8,9}{
    \draw (11.9,\x/2) -- (12.1,\x/2);
    \ifodd\xi

        \node[anchor=east] at (0,\x/2) {$\x$};
    \fi
\draw (12+1/2,1pt) -- (12+1/2,-1pt);
 \draw (12+2/2,1pt) -- (12+2/2,-1pt);
 \draw (12+3/2,1pt) -- (12+3/2,-1pt);
 \draw (12+4/2,1pt) -- (12+4/2,-1pt);
     \node[anchor=north] at (12+1/2,0) {$1$};
     \node[anchor=north] at (12+2/2,0) {$2$};
     \node[anchor=north] at (12+3/2,0) {$3$};
     \node[anchor=north] at (12+4/2,0) {$4$};
}

\
\foreach \point in {(6+1/2,1/2),(6+1/2,2/2),(6+1/2,3/2),(6+1/2,4/2),(6+1/2,5/2),(6+1/2,6/2),
(6+1/2,7/2),(6+1/2,8/2),
(6+2/2,2/2),(6+2/2,3/2),(6+2/2,4/2),(6+2/2,5/2),(6+2/2,6/2),
(6+3/2,3/2),(6+3/2,4/2),(6+3/2,5/2)}{
    \fill \point circle (2pt);
}

\node at (3,4) {$n=8$};
\node at (9,4) {$n=9$};
\node at (15,4) {$n=10$};
\foreach \point in {
(12+1/2,1/2),
(12+1/2,2/2),(12+1/2,3/2),(12+1/2,4/2),(12+1/2,5/2),(12+1/2,6/2),
(12+1/2,7/2),(12+1/2,8/2),(12+1/2,9/2),
(12+2/2,2/2),(12+2/2,3/2),(12+2/2,4/2),(12+2/2,5/2),
(12+2/2,6/2),(12+2/2,7/2),
(12+3/2,3/2),(12+3/2,4/2),(12+3/2,5/2)}{
    \fill \point circle (2pt);
}

\end{tikzpicture}

\caption{Comparing the v-number $\v(I)$ and the degree of $h_{R/I}(t)$
  for all connected graphs on $n \in \{2,\ldots,10\}$ vertices when $I$ is an edge ideal.}\label{fig:comparevandd}
\end{figure}
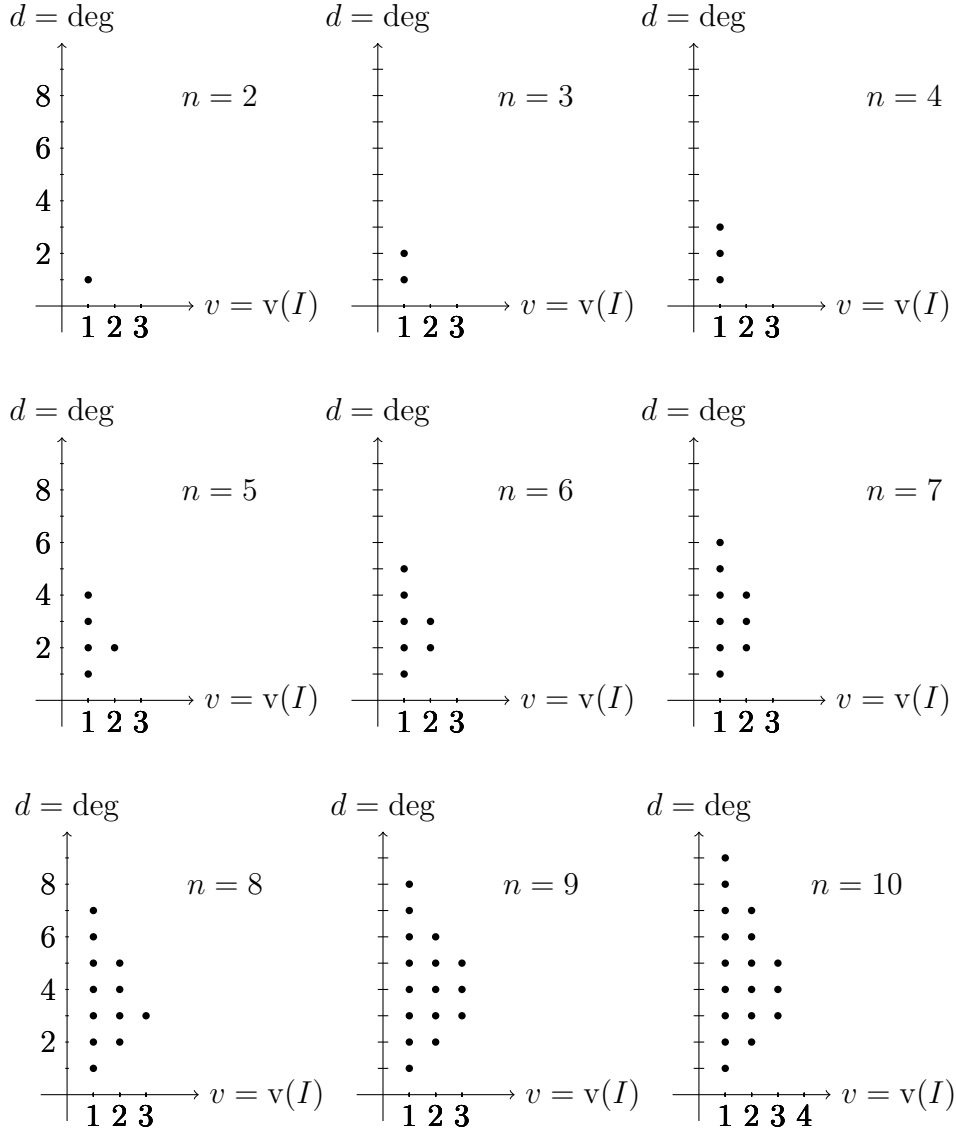
\newpage
\newpage

\section*{Appendix B: Macaulay2 code}

Below is {\it Macaulay2} code for the graphs of \Cref{fig:11vertex}
and \Cref{fig:secondminex}. We present this code here so that the reader can 
cut-and-paste these graphs into their own projects. We have also included the commands for computing the v-number and the reduced Hilbert
series.
\footnotesize
\begin{verbatim}
loadPackage "CodingTheory"
R:=QQ[x_1..x_11]

--Edge ideals of graphs with v-number = 3, deg h-poly = 2
i=monomialIdeal(x_1*x_2,x_1*x_6,x_1*x_7,x_1*x_5,
    x_2*x_8,x_2*x_10,x_2*x_9,x_2*x_5,x_3*x_5,x_3*x_4,
    x_3*x_10,x_3*x_11,x_4*x_5,x_4*x_6,x_4*x_8,x_4*x_9,
    x_6*x_7,x_6*x_8,x_6*x_10,x_7*x_9,x_7*x_11,x_8*x_9,
    x_8*x_10,x_9*x_11,x_10*x_11)

j=monomialIdeal(x_1*x_4,x_1*x_5,x_1*x_8,x_1*x_9,
    x_2*x_5,x_2*x_6,x_2*x_8,x_2*x_10,x_2*x_11,
    x_3*x_6,x_3*x_9,x_3*x_7,x_3*x_10,
    x_4*x_7,x_4*x_8,x_4*x_11,x_5*x_9,x_5*x_10,
    x_5*x_11,x_6*x_8,x_6*x_9,x_6*x_11,x_7*x_10,
    x_7*x_11,x_9*x_11)
    
vNumber(i)
vNumber(j)
s = hilbertSeries(i)
t = hilbertSeries(j)
reduceHilbert s
reduceHilbert t
\end{verbatim}
\normalsize
\subsection*{Acknowledgments}
Part of this work was made possible by the facilities of the Shared Hierarchical Academic Research Computing Network ({\tt SHARCNET:www.sharcnet.ca}) and Digital Research Alliance of Canada ({\tt https://alliancecan.ca/en}).
In addition, we made use of {\it Macaulay2} \cite{M2} and the packages
{\tt CodingTheory} \cite{CTM2} and {\tt Nauty} \cite{nauty}. Van Tuyl's research was supported in part by NSERC Discovery Grant 2024--05299. He  would also like to thank the hospitality of the
Fields Institute in Toronto, Canada.  

\bibliographystyle{amsplain}
\bibliography{sample}

\begin{bibdiv}
\begin{biblist}

\bib{ass23}{article}{
      author={Ambhore, Siddhi~Balu},
      author={Saha, Kamalesh},
      author={Sengupta, Indranath},
       title={The {${\rm v}$}-number of binomial edge ideals},
        date={2024},
     journal={Acta Math. Vietnam.},
      volume={49},
      number={4},
       pages={611\ndash 628},
         url={https://doi.org/10.1007/s40306-024-00540-w},
}

\bib{CTM2}{article}{
      author={Ball, Taylor},
      author={Camps, Eduardo},
      author={Chimal-Dzul, Henry},
      author={Jaramillo-Velez, Delio},
      author={L\'opez, Hiram},
      author={Nichols, Nathan},
      author={Perkins, Matthew},
      author={Soprunov, Ivan},
      author={Vera-Mart\'inez, German},
      author={Whieldon, Gwyn},
       title={Coding theory package for {M}acaulay2},
        date={2021},
     journal={J. Softw. Algebra Geom.},
      volume={11},
      number={1},
       pages={113\ndash 122},
         url={https://doi.org/10.2140/jsag.2021.11.113},
}

\bib{bvt2023}{article}{
      author={Bhaskara, Kieran},
      author={Van~Tuyl, Adam},
       title={Comparing invariants of toric ideals of bipartite graphs},
        date={2023},
     journal={Proc. Amer. Math. Soc. Ser. B},
      volume={10},
       pages={219\ndash 232},
         url={https://doi.org/10.1090/bproc/174},
}

\bib{bkoss24}{article}{
      author={Biermann, Jennifer},
      author={Kara, Selvi},
      author={O'Keefe, Augustine},
      author={Skelton, Joseph},
      author={Sosa, Gabriel},
       title={Degree of $h$-polynomials of edge ideals},
        date={2024},
     journal={Preprint},
       pages={{\tt arXiv:2408.12544}},
}

\bib{bm23}{article}{
      author={Biswas, Prativa},
      author={Mandal, Mousumi},
       title={A study of ${\textrm{v}}$-number for some monomial ideals},
        date={2024},
     journal={Collect. Math.},
      eprint={https://doi.org/10.1007/s13348-024-00451-x},
}

\bib{bms24}{article}{
      author={Biswas, Prativa},
      author={Mandal, Mousumi},
      author={Saha, Kamalesh},
       title={Asymptotic behaviour and stability index of v-numbers of graded ideals},
        date={2024},
     journal={Preprint},
       pages={{\tt arXiv:2402.16583}},
}

\bib{CHHVT}{book}{
      author={Carlini, Enrico},
      author={H\`a, Huy~T\`ai},
      author={Harbourne, Brian},
      author={Van~Tuyl, Adam},
       title={Ideals of powers and powers of ideals},
      series={Lecture Notes of the Unione Matematica Italiana},
   publisher={Springer, Cham},
        date={2020},
      volume={27},
         url={https://doi.org/10.1007/978-3-030-45247-6},
}

\bib{civan23}{article}{
      author={Civan, Yusuf},
       title={The {$v$}-number and {C}astelnuovo-{M}umford regularity of graphs},
        date={2023},
     journal={J. Algebraic Combin.},
      volume={57},
      number={1},
       pages={161\ndash 169},
         url={https://doi.org/10.1007/s10801-022-01164-9},
}

\bib{concav23}{article}{
      author={Conca, Aldo},
       title={A note on the {$v$}-invariant},
        date={2024},
     journal={Proc. Amer. Math. Soc.},
      volume={152},
      number={6},
       pages={2349\ndash 2351},
         url={https://doi.org/10.1090/proc/16767},
}

\bib{nauty}{article}{
      author={Cook, David, II},
       title={Nauty in {M}acaulay2},
        date={2011},
     journal={J. Softw. Algebra Geom.},
      volume={3},
       pages={1\ndash 4},
         url={https://doi.org/10.2140/jsag.2011.3.1},
}

\bib{cstpv20}{article}{
      author={Cooper, Susan~M.},
      author={Seceleanu, Alexandra},
      author={Tohuaneanu, Stefan~O.},
      author={Pinto, Maria~Vaz},
      author={Villarreal, Rafael~H.},
       title={Generalized minimum distance functions and algebraic invariants of {G}eramita ideals},
        date={2020},
     journal={Adv. in Appl. Math.},
      volume={112},
       pages={101940, 34},
         url={https://doi.org/10.1016/j.aam.2019.101940},
}

\bib{dhs13}{article}{
      author={Dao, Hailong},
      author={Huneke, Craig},
      author={Schweig, Jay},
       title={Bounds on the regularity and projective dimension of ideals associated to graphs},
        date={2013},
     journal={J. Algebraic Combin.},
      volume={38},
      number={1},
       pages={37\ndash 55},
         url={https://doi.org/10.1007/s10801-012-0391-z},
}

\bib{djs25}{article}{
      author={Dey, Deblina},
      author={Jayanthan, A.~V.},
      author={Saha, Kamalesh},
       title={On the {${\rm v}$}-number of binomial edge ideals of some classes of graphs},
        date={2025},
     journal={Internat. J. Algebra Comput.},
      volume={35},
      number={1},
       pages={119\ndash 143},
         url={https://doi.org/10.1142/S0218196724500607},
}

\bib{fkvt2020}{article}{
      author={Favacchio, Giuseppe},
      author={Keiper, Graham},
      author={Van~Tuyl, Adam},
       title={Regularity and {$h$}-polynomials of toric ideals of graphs},
        date={2020},
     journal={Proc. Amer. Math. Soc.},
      volume={148},
      number={11},
       pages={4665\ndash 4677},
         url={https://doi.org/10.1090/proc/15126},
}

\bib{fs23}{article}{
      author={Ficarra, Antonino},
      author={Sgroi, Emanuele},
       title={Asymptotic behaviour of the $\text{v}$-number of homogeneous ideals},
        date={2023},
     journal={Preprint},
       pages={{\tt arxiv:2306.14243}},
}

\bib{fs24}{article}{
      author={Ficarra, Antonino},
      author={Sgroi, Emanuele},
       title={Asymptotic behavior of integer programming and the v-function of a graded filtration},
        date={2025},
     journal={J. Algebra Appl.},
      eprint={https://doi.org/10.1142/S0219498826502361},
}

\bib{fd24}{article}{
      author={Fiorindo, Luca},
      author={Ghosh, Dipankar},
       title={On the asymptotic behavior of the {V}asconcelos invariant for graded modules},
        date={2025},
     journal={Nagoya Math. J.},
       pages={1\ndash 15},
}

\bib{gp25}{article}{
      author={Ghosh, Dipankar},
      author={Pramanik, Siddhartha},
       title={Asymptotic v-numbers of graded (co)homology modules involving powers of an ideal},
        date={2025},
     journal={J. Algebra},
      volume={671},
       pages={61\ndash 74},
}

\bib{M2}{article}{
      author={Grayson, Daniel~R.},
      author={Stillman, Michael~E.},
       title={Macaulay2, a software system for research in algebraic geometry},
      eprint={http://www2.macaulay2.com},
}

\bib{grv21}{article}{
      author={Grisalde, Gonzalo},
      author={Reyes, Enrique},
      author={Villarreal, Rafael~H.},
       title={Induced matchings and the v-number of graded ideals},
        date={2021},
     journal={Mathematics},
      volume={9},
      number={Paper No. 22},
       pages={2860},
}

\bib{havan08}{article}{
      author={H\`a, Huy~T\`ai},
      author={Van~Tuyl, Adam},
       title={Monomial ideals, edge ideals of hypergraphs, and their graded {B}etti numbers},
        date={2008},
     journal={J. Algebraic Combin.},
      volume={27},
      number={2},
       pages={215\ndash 245},
         url={https://doi.org/10.1007/s10801-007-0079-y},
}

\bib{hkkmvt2021}{article}{
      author={Hibi, Takayuki},
      author={Kanno, Hiroju},
      author={Kimura, Kyouko},
      author={Matsuda, Kazunori},
      author={Van~Tuyl, Adam},
       title={Homological invariants of {C}ameron-{W}alker graphs},
        date={2021},
     journal={Trans. Amer. Math. Soc.},
      volume={374},
      number={9},
       pages={6559\ndash 6582},
         url={https://doi.org/10.1090/tran/8416},
}

\bib{hkmvt2022}{article}{
      author={Hibi, Takayuki},
      author={Kimura, Kyouko},
      author={Matsuda, Kazunori},
      author={Van~Tuyl, Adam},
       title={The regularity and {$h$}-polynomial of {C}ameron-{W}alker graphs},
        date={2022},
     journal={Enumer. Comb. Appl.},
      volume={2},
      number={3},
       pages={Paper No. S2R17, 12},
         url={https://doi.org/10.54550/eca2022v2s3r17},
}

\bib{hm2018}{article}{
      author={Hibi, Takayuki},
      author={Matsuda, Kazunori},
       title={Regularity and {$h$}-polynomials of monomial ideals},
        date={2018},
     journal={Math. Nachr.},
      volume={291},
      number={16},
       pages={2427\ndash 2434},
         url={https://doi.org/10.1002/mana.201700476},
}

\bib{hm2022}{article}{
      author={Hibi, Takayuki},
      author={Matsuda, Kazunori},
       title={Regularity and {$h$}-polynomials of binomial edge ideals},
        date={2022},
     journal={Acta Math. Vietnam.},
      volume={47},
      number={1},
       pages={369\ndash 374},
         url={https://doi.org/10.1007/s40306-021-00416-3},
}

\bib{hmv19}{article}{
      author={Hibi, Takayuki},
      author={Matsuda, Kazunori},
      author={Van~Tuyl, Adam},
       title={Regularity and {$h$}-polynomials of edge ideals},
        date={2019},
     journal={Electron. J. Combin.},
      volume={26},
      number={1},
       pages={Paper No. 1.22, 11},
         url={https://doi.org/10.37236/8247},
}

\bib{ht10}{article}{
      author={Hoa, Le~Tuan},
      author={Tam, Nguyen~Duc},
       title={On some invariants of a mixed product of ideals},
        date={2010},
     journal={Arch. Math. (Basel)},
      volume={94},
      number={4},
       pages={327\ndash 337},
         url={https://doi.org/10.1007/s00013-010-0112-6},
}

\bib{v-edge}{article}{
      author={Jaramillo, Delio},
      author={Villarreal, Rafael~H.},
       title={The v-number of edge ideals},
        date={2021},
     journal={J. Combin. Theory Ser. A},
      volume={177},
       pages={Paper No. 105310, 35},
         url={https://doi.org/10.1016/j.jcta.2020.105310},
}

\bib{js24_v-binom}{article}{
      author={Jaramillo-Velez, Delio},
      author={Seccia, Lisa},
       title={Connected domination in graphs and v-numbers of binomial edge ideals},
        date={2024},
     journal={Collect. Math.},
      volume={75},
      number={3},
       pages={771\ndash 793},
         url={https://doi.org/10.1007/s13348-023-00412-w},
}

\bib{kmt25}{article}{
      author={Kataoka, Tatsuya},
      author={Muta, Yuji},
      author={Terai, Naoki},
       title={The v-numbers of {S}tanley-{R}eisner ideals from the viewpoint of {A}lexander dual complexes},
        date={2025},
     journal={Preprint},
       pages={{\tt arXiv:2504.07535}},
}

\bib{kns25}{article}{
      author={Kumar, Manohar},
      author={Nanduri, Ramakrishna},
      author={Saha, Kamalesh},
       title={The slope of the {${\rm v}$}-function and the {W}aldschmidt constant},
        date={2025},
     journal={J. Pure Appl. Algebra},
      volume={229},
      number={2},
       pages={Paper No. 107881, 13},
         url={https://doi.org/10.1016/j.jpaa.2025.107881},
}

\bib{sahacover23}{article}{
      author={Saha, Kamalesh},
       title={The v-number and {C}astelnuovo-{M}umford regularity of cover ideals of graphs},
        date={2024},
     journal={Int. Math. Res. Not. IMRN},
      number={11},
       pages={9010\ndash 9019},
         url={https://doi.org/10.1093/imrn/rnad277},
}

\bib{saha24_binomexpan}{article}{
      author={Saha, Kamalesh},
       title={Binomial expansion and the $\mathrm{v}$-number},
        date={2025},
     journal={To appear in Commun. Algebra},
      eprint={https://arxiv.org/abs/2406.05567},
}

\bib{ksvgor23}{article}{
      author={Saha, Kamalesh},
      author={Kotal, Nirmal},
       title={On the v-number of {G}orenstein ideals and {F}robenius powers},
        date={2024},
     journal={Bull. Malays. Math. Sci. Soc.},
      volume={47},
      number={6},
       pages={Paper No. 167, 17},
         url={https://doi.org/10.1007/s40840-024-01763-8},
}

\bib{ssvmon22}{article}{
      author={Saha, Kamalesh},
      author={Sengupta, Indranath},
       title={The {${v}$}-number of monomial ideals},
        date={2022},
     journal={J. Algebraic Combin.},
      volume={56},
      number={3},
       pages={903\ndash 927},
         url={https://doi.org/10.1007/s10801-022-01137-y},
}

\bib{stanley96}{book}{
      author={Stanley, Richard~P.},
       title={Combinatorics and commutative algebra},
     edition={Second},
      series={Progress in Mathematics},
   publisher={Birkh\"auser Boston, Inc., Boston, MA},
        date={1996},
      volume={41},
}

\bib{as24}{article}{
      author={Vanmathi, A.},
      author={Sarkar, Parangama},
       title={v-numbers of symbolic power filtrations},
        date={2025},
     journal={Collect. Math.},
      eprint={https://doi.org/10.1007/s13348-025-00476-w},
}

\end{biblist}
\end{bibdiv}

\end{document}